\documentclass[a4paper,11pt]{article} 
\usepackage{amsmath}                  
\usepackage{amssymb}                  
\usepackage{amsthm}                   
\usepackage{geometry}                 
\usepackage{graphicx}                 
\usepackage{float}                    
\usepackage{caption}                  
\usepackage{textcomp}
\usepackage{multirow}
\usepackage{lscape}
\usepackage{rotating}
\usepackage{pdflscape}
\usepackage{datetime}    
\usepackage[english]{babel}

\newcommand{\bec}{\mbox{$\boldsymbol{e}$}}

\newtheorem{theorem}{Theorem}
\newtheorem{lemma}{Lemma}
\newtheorem{example}{Example}
\newtheorem{remark}{Remark}

\makeatletter

\par
\begin{document}
\newcommand{\bea}{\begin{eqnarray}}
\newcommand{\eea}{\end{eqnarray}}
\newcommand{\nn}{\nonumber}
\newcommand{\bee}{\begin{eqnarray*}}
\newcommand{\eee}{\end{eqnarray*}}
\newcommand{\lb}{\label}
\newcommand{\la}{\lambda}
\newcommand{\pr}{\partial}
\newcommand{\bp}{\mbox{$\mathbf{p}$}}
\newcommand{\bpi}{\mbox{\boldmath$\pi$}}
\newcommand{\bphi}{\mbox{\boldmath$\phi$}}
\newcommand{\bPhi}{\mbox{\boldmath$\Phi$}}
\newcommand{\bpsi}{\mbox{\boldmath$\psi$}}
\newcommand{\be}{\mbox{$\boldsymbol{e}$}}
\newcommand{\nii}{\noindent}
\title{{\bf{ Analysis of an infinite-buffer batch-size-dependent service queue
with discrete-time Markovian arrival process:}} D-$MAP/G_n^{(a,b)}/1$}
\author{U. C. Gupta$^{1}$\thanks{Corresponding author.
\newline
\textit{E-mail addresses:}
umesh@maths.iitkgp.ac.in (U. C. Gupta),
nitinkumar7276@gmail.com (Nitin Kumar),
\newline
spiitkgp11@gmail.com (S. Pradhan),
faridaparvezb@gmail.com (F. P. Barbhuiya).
},  Nitin Kumar$^{1}$, S. Pradhan$^{2}$, F. P. Barbhuiya$^{2}$~\\
{\it $^{1}$Department of Mathematics, Indian Institute of Technology Kharagpur,}\\
{\it Kharagpur-721302, India.}\\
{\it $^{2}$Department of Mathematics, Visvesvaraya National Institute of Technology Nagpur,}\\
{\it Nagpur-440010, India.}}
\date{\today}
%
\maketitle \nii{\bf Abstract:} Discrete-time queueing models find huge applications as they are used in
modeling queueing systems arising in digital platforms like telecommunication systems, computer networks,
etc. In this paper, we analyze an infinite-buffer queueing model with discrete Markovian arrival process. The
units on arrival are served in batches by a single server according to the general bulk-service rule, and the
service time follows general distribution with service rate depending on the size of the batch being served.
We mathematically formulate the model using the supplementary variable technique and obtain the vector
generating function at the departure epoch. The generating function is in turn used to extract the joint
distribution of queue and server content in terms of the roots of the characteristic equation. Further, we
develop the relationship between the distribution at the departure epoch and the distribution at arbitrary,
pre-arrival and outside observer's epoch, which is used to obtain the latter ones. We evaluate some essential
performance measures of the system and also discuss the computing process extensively which is demonstrated
by few numerical examples.
%
\vspace*{.2cm}\\
{\bf Keywords:} Batch-size dependent, Discrete-time, General bulk service, Infinite-buffer, Markovian arrival
process, Phase-type.
\section{Introduction}
Queueing models involving batch service have been investigated by
many researchers in the past due to its potential application in
several stochastic systems. Chaudhry and Templeton
\cite{Chaudhry1983} and Medhi \cite{medhi1984recent}  provides a
detailed discussion on different types of bulk queueing models. The
general bulk service rule find application in the field of
manufacturing and production systems, where the server starts
service with a batch of minimum threshold size `$a$' and a maximum
size `$b$'. Moreover, the instances when the service rate (or
service time) is dependent on the size of the batch being served,
are more appropriate in modeling many of the real world problems.
Such queues are known as batch size dependent service queues and
plays a vital role in group screening practice of blood or urine
samples examined for a particular disease, say HIV (see Abolnikov
and Dukhovny \cite{abolnikov2003optimization}, Bar-lev et al.
\cite{bar2004optimal, bar2007applications}). A group if found
infected by the disease is set for further testing which may occur
individually. If the size of the batch is large, then testing
individual blood sample may take longer time which is a direct
application to batch size dependent service. Moreover, in modern
telecommunication systems, the transfer of information (data, voice,
videos, images, etc.) occurs in packets in bulk where the
transmission time depends on the batch size of packets. In the past
few years many researchers have focussed on studying batch size
dependent service queues, both with finite-buffer (see Yu and Alfa
\cite{yu2015algorithm}, Banerjee et al. \cite{banerjee2014analysis})
and infinite-buffer (see Claeys et al. \cite{claeys2013analysis,
claeys2013tail}, Pradhan and Gupta \cite{pradhan2017analysis,
pradhan2017modeling}). Claeys et al. \cite{claeys2013tail} provided
the application of batch-size dependent service policy mainly in the
area of telecommunication sector and illustrated the effect of
neglecting batch-size dependent service times on the performance
measures of the system.
\par In many real-world queueing systems the arrival of customers or units do
not occur independently of each other. As for instance, in telecommunication systems the transmission of
information, in the form of packetized data, takes place with a very high speed over a large network which
exhibits burstiness, correlation and self-similarity. These features cannot be captured well using the
traditional Poisson or Bernoulli arrival processes, and hence Markovian arrival process ($MAP$) can be
adopted to cope with the bursty and correlated nature of the arrival process. In particular, the
discrete-time analogous of $MAP$, i.e., D-$MAP$ is more applicable in telecommunication context due to the
discrete nature of the transmission of information units in slotted systems, see for example Alfa
\cite{alfa2010queueing,alfa2016applied}, Bruneel and Kim \cite{bruneel1993discrete}, Hunter
\cite{hunter1983mathematical}, Takagi \cite{takagi1993queuing}, Woodward \cite{woodward1994communication}.
D-$MAP$ is also a versatile arrival process and covers many other well known arrival processes such as the
Bernoulli arrival process, the switched Bernoulli process (SBP), the Markov modulated Bernoulli process
(MMBP), the discrete-time PH-renewal process etc. Much work has been done in the past on queueing models with
D-$MAP$ arrival process with both finite and infinite-buffer. As for instance, Chaudhry and Gupta
\cite{chaudhry2003queue, chaudhry2003analysis} studied the finite-buffer D-$MAP/G/1/N$ and
D-$MAP/G^{(a,b)}/1/N$ queue respectively where they obtained the queue-content distribution at various
epochs. Further, Gupta et al \cite{gupta2014analysis} addressed a more general D-$MAP/G_n/1/N$ queue with the
service time depending on the number of customers waiting in the queue. Yu and Alfa \cite{yu2015algorithm}
considered a batch size dependent service D-${MAP/G^{(1, a, b)}_n/1/N}$ queue where the server serves the
customers individually if there are less than `$a$' customers in the queue, otherwise it servers according to
the general bulk service ($a,b$) rule. For the infinite-buffer queue, Pradhan and Gupta
\cite{pradhan2017analysis} addressed the continuous-time $MAP/G^{(a,b)}_n/1$ queue whereas Claeys et al.
\cite{claeys2013analysis} studied the discrete time analogous of \cite{pradhan2017analysis} with the
assumption of batch Markovian arrival process.
\par In this paper, we give a complete theoretical and computational
analysis of an infinite-buffer discrete-time queueing model with Markovian arrival process (D-$MAP$). We
assume that the server provides service in batches according to the general bulk service rule and the service
time follows general distribution and depends on the size of the batch undergoing service. At first, using
the supplementary variable technique, we obtain the steady-state bivariate vector generating function (VGF)
of the queue-length and server content at the departure epoch of the batch, in a completely known form. From
the bivariate VGF we extract the distribution at the departure epoch in terms of roots of the associated
characteristic equation. Further, in order to obtain the distribution at arbitrary, pre-arrival and outside
observer's epoch, we establish their relation with the distribution at the departure epoch. Keeping note of
the complexity of the model under consideration, we discuss in detail the whole computing process by
considering discrete phase-type and arbitrary distributed service time distributions which cover almost all
distributions that arise in various applications. We evaluate some significant performance characteristics of
the model and demonstrate the computing process by certain numerical examples. In this paper, the use of
supplementary variable technique makes the analysis of the model relatively simpler, which otherwise, would
have been difficult using the embedded Markov chain technique because of the complexity associated with the
construction of the transition probability matrix. It may be mentioned that Claeys et al.
\cite{claeys2013analysis} considered the batch arrival D-$BMAP/G_n^{(l,c)}/1$ queue; however, they focussed
mainly on obtaining the VGF and moments whereas in this paper we focus on extracting the distributions in a
completely known form from the VGF.
\par The remaining portion of the paper is organized as follows. In
Section \ref{sec2} we give the detailed description of the
considered discrete-time system followed by the analysis of the
model in Section \ref{sec3}. In Section \ref{sec4} we obtain the
joint queue and server content distribution at various epochs and
then discuss the computing process in detail in Section \ref{sec5}.
In Section \ref{sec6} we discuss some special cases of the model and
evaluate some performance measures in Section \ref{sec7}. We present
some numerical examples in Section \ref{sec8} which is followed by
the conclusion.
%
%
\section{Model description, assumptions and notations}\label{sec2}
We consider a discrete-time queueing model in which the customers arrive according to discrete-Markovian
arrival process(D-MAP) and service time of the batches of customers follow general distribution. Below we
describe various processes.
\begin{itemize}
\item \textbf{Arrival process:} In D-MAP the arrivals are governed by an underlying $m$-state Markov
chain having probability $C_{ij},(1 \le i,j \le m)$ with a
transition from state $i$ to $j$ without an arrival, and having
probability  $D_{ij},(1 \le i,j \le m)$ with a transition from state
$i$ to $j$ with an arrival. Let
$\mathbf{C}=[C_{ij}],\mathbf{D}=[D_{ij}]$ be the $m \times m$
non-negative matrices both having at least one positive entry. The
matrix $\mathbf{(C+D)}$ with  $\mathbf{(C+D)}\boldsymbol{e}=\be$,
where $\be$ is a column vector of ones with suitable dimension, is a
stochastic matrix corresponding to an irreducible Markov chain
underlying the D-MAP. Let
$\overline{\bpi}=[\overline{\pi}_{1},\overline{\pi}_{2},\dots,\overline{\pi}_{m}]$
be the stationary probability vector of the underlying Markov chain,
i.e.,
$\overline{\bpi}\mathbf{(C+D)}=\overline{\bpi},~\overline{\bpi}\be=1.$
The fundamental and stationary arrival rate is given by
$(\lambda^{*})=\overline{\bpi}\mathbf{D}\be$.
\item \textbf{Service rule:} A single server serves the customers
in batches according to general bulk service $(a, b)$ rule. If the
queue contains less than `$a$' number ofcustomers, server enters
into the idle period and waits to initiate the service until at
least `$a$' customers gets accumulated. When the queue size is $r
(a \le r \le b)$, entire group of customers are taken for service.
However, if the queue size is greater than `$b$', the server serves
first `$b$' customers and the remaining have to wait for the next
round of service.
\item \textbf{Service process:} The service time of the batches follow
general distribution and are assumed to be dependent on batch size of ongoing service. Let us define the
random variable ${V_r} (a \le r \le b)$ as the service time of a batch of size $r$ with probability mass
function $s_r(n)=Pr(V_r=n), ~n \ge 1$, probability generating function $S^*_r (z)=
\sum\limits_{n=1}^{\infty}s_r(n) z^n$, and the mean service time
$\mu^{-1}_r=S_r=\sum\limits_{n=1}^{\infty}ns_r(n)=S^{*(1)}_r(1)$, where $S^{*(1)}_r(1)=\frac{d}{dz}S^*_r
(z)\Big|_{z=1}$.
\item \textbf{Late arrival system with delayed access:} In discrete-time,
the time axis is divided into intervals of equal length referred to
as (time) slots, separated by slot boundaries. We assume that the
length of each slot is unity and time axis is  marked as $0,1,2,
\dots ,t ,\dots $. We further assume that a potential arrival occurs
in the interval $(t-,t)$ and a potential departure takes place in
the interval $(t,t+)$. However, if an arrival finds the server idle,
it cannot depart in the same slot in which it has arrived and has to
wait for at least one slot before getting served. This is referred
as late arrival system with delayed access (LAS-DA), see Hunter
\cite{hunter1983mathematical}. Various time epochs at which events
occur are delineated in Figure \ref{LASFig}.
\item For the stability of the system, we must have that $\rho<1$
where $\rho=\frac{\lambda^{*}}{b \mu_b}$.
\end{itemize}
\begin{figure}
 \centering
       \includegraphics[width=15cm, height=10cm]{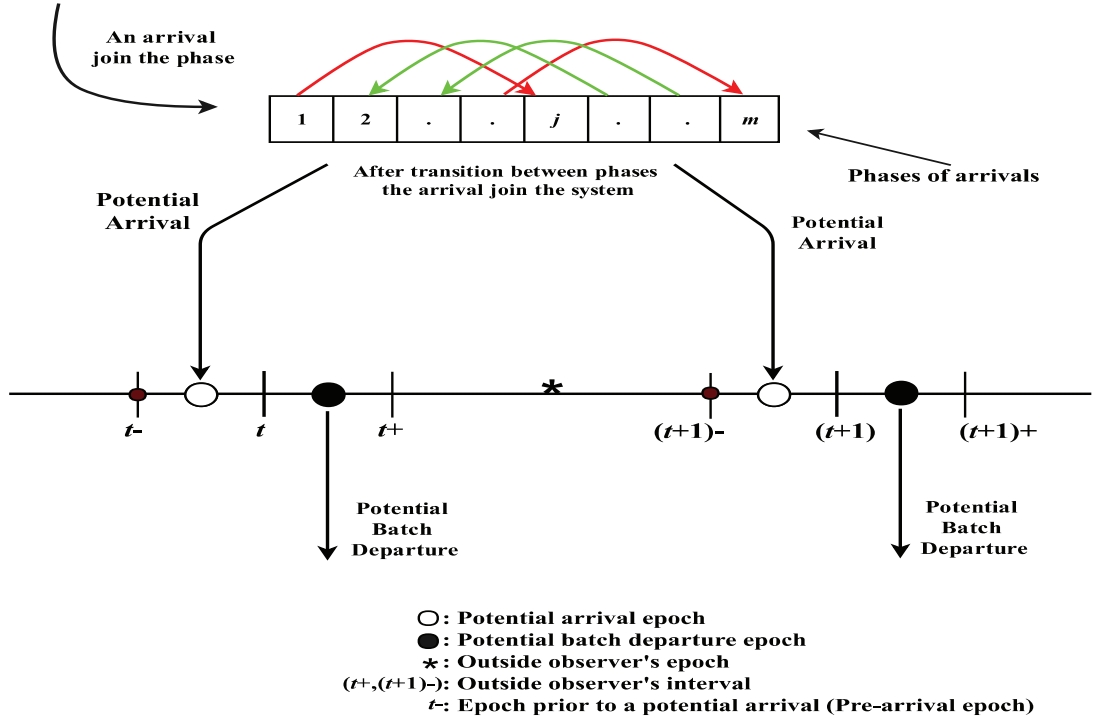}
     \caption{Various time epoch in LAS-DA}  \label{LASFig}
\end{figure}
\par
Let us denote the matrix $\mathbf{A}_r(n,k)$, $a \le r \le b,$ to be
the matrix of order $m \times m$ whose $(i,j)$th element is the
conditional probability that, a departure which left at least $a$
customers in the queue with the arrival process in state $i$,
exactly $n$ new customers arrive during the service period (of
length $k$ slots) of $r$ customers; the phase of the arrival process
is in phase $j$ at the departure epoch. So $\mathbf{A}_r(n,k)$ can
be written as \bea
\mathbf{A}_r(0,k) &=& \mathbf{C}   \mathbf{A}_r(0,k-1), ~~ k \ge 1,  \nn \\
\mathbf{A}_r(n,k) &=& \mathbf{C}   \mathbf{A}_r(n,k-1) + \mathbf{D}
\mathbf{A}_r(n-1,k-1), ~~ k\ge n \ge 1, \nn \eea with
$\mathbf{A}_r(0,0)=\mathbf{I}$ and $\mathbf{A}_r(n,k) =\mathbf{0},~n
> k \ge 0,$ where $\mathbf{I}$ and $\mathbf{0}$ are identity and zero
matrix of order $m \times m$, respectively. Also, let $\mathbf{A}^*_r(z,k)$ be the matrix-generating function
of $\mathbf{A}_r(n,k)$, then \bea \mathbf{A}^*_r(z,k) &=&\sum_{n=0}^{\infty} \mathbf{A}_r(n,k)
z^n=[\mathbf{A}^*_r(z,1)]^k=[\mathbf{C}+\mathbf{D}z]^k, \lb{03} \eea where
$\mathbf{A}^*_r(z,1)=\mathbf{C}+\mathbf{D}z$ is the matrix-generating function of the number of customers
arriving in one slot. Let us denote the matrix $\mathbf{A}_r(n)=[A_r(n)]_{i,j}$, $n \ge 0, a \le r \le b, 1
\le i, j \le m,$ to be the conditional probability that, a departure which left at least $a$ customers in
the queue with the arrival process in state $i$, and during the service period of $r$ customers exactly
$n$ new customers arrive; the phase of the arrival process is in phase $j$ at the departure epoch. So
$\mathbf{A}_r(n)$ can be written as \bea \mathbf{A}_r(n) &=& \sum_{k=\max(1,n)}^{\infty} s_r(k)
\mathbf{A}_r(n,k), ~~ n \ge 0. \nn \eea Further, let $\mathbf{A}^*_r(z)$ be the matrix-generating function of
$\mathbf{A}_r(n)$. Therefore we have
\bea \mathbf{A}^*_r(z) &=& \sum_{n=0}^{\infty}\mathbf{A}_r(n) z^n=\sum_{k=1}^{\infty} s_r(k)
[\mathbf{C}+\mathbf{D}z]^k,~~~~  a \le r \le b. \label{FA(s)} \eea
\begin{remark}
As and when more notations are used, they will be defined at respective places.
\end{remark}
\section{Analysis of the model} \label{sec3}
Let us define the following random variables at the beginning of the
slot boundary $i.e.$ just before a potential arrival:
\begin{itemize}
  \item $N_q (t-) \equiv$  Number of customers in the queue waiting for service at $t-$.
  \item $N_s (t-) \equiv$ Number of customers with the server at $t-$.
  \item $J (t-) \equiv$ Phase of the arrival process at $t-$.
  \item $U (t-) \equiv $ Remaining service time of a batch in service (if any) at $t-$.
\end{itemize}
\nii
Let us define for $1 \le i \le m$,
\bea
p_{i}(n,0;t-)     &=& Pr\{N_q (t-) = n, N_s (t-) = 0, J(t-) = i, \hbox{server idle} \}, ~~~ 0 \le n \le a-1, \nn \\
\pi_{i}(n,r,u;t-) &=& Pr\{N_q (t-) = n, N_s (t-) = r, J(t-) = i,  U(t-)=u, \hbox{server busy} \}, \nn\\
&& \hspace{6cm} ~ n \ge 0, ~a \le r \le b, ~ u \ge 1. \nn
\eea
Also let us define the limiting probabilities as
\bea
p_{i}(n,0)     &=& \lim_{t-~ \rightarrow \infty} p_{i}(n,0;t-),     \nn \\
\pi_{i}(n,r,u) &=& \lim_{t-~ \rightarrow \infty} \pi_{i}(n,r,u;t-).  \nn \eea
Further, we define \bea
\bp(n,0)     &=& [p_{1}(n,0),\dots,p_{i}(n,0),\dots,p_{m}(n,0)],     \nn \\
\bpi(n,r,u)     &=& [\pi_{1}(n,r,u),\dots,\pi_{i}(n,r,u),\dots,\pi_{m}(n,r,u)]. \nn
\eea
Relating the states of the system at two consecutive time epochs $t-$ and $(t+1)-$ for each phase, and then
using matrix and vector notations after taking $t-~ \to \infty$, we obtain in steady-state
\bea
\bp(0,0) &=& \bp(0,0) \mathbf{C} + \sum_{r=a}^{b} \bpi(0,r,1)\mathbf{C}, \lb{sp1}\\
\bp(n,0) &=& \bp(n,0) \mathbf{C} + \sum_{r=a}^{b} \bpi(n,r,1)\mathbf{C} + \bp(n-1,0) \mathbf{D} +
\sum_{r=a}^{b} \bpi(n-1,r,1)\mathbf{D},\nn\\
&&\hspace{9cm} ~~ 1 \le n \le a-1, \lb{sp2}\\
\bpi(0,a,u)&=&\bpi(0,a,u+1)\mathbf{C} + \bp(a-1,0) \mathbf{D} s_a(u) + \sum_{r=a}^{b} \bpi(a,r,1) \mathbf{C}
s_a(u) \nn \\
&&~~+ \sum_{r=a}^{b} \bpi(a-1,r,1) \mathbf{D} s_a(u), \lb{sp3}\\
\bpi(0,r,u)&=&\bpi(0,r,u+1)\mathbf{C} + \sum_{i=a}^{b} \bpi(r,i,1) \mathbf{C} s_r(u) + \sum_{i=a}^{b} \bpi(r-1,i,1) \mathbf{D} s_r(u),\nn \\
&&\hspace{9cm} ~~ a+1 \le r \le b, \lb{sp4}\\
\bpi(n,r,u)&=&\bpi(n,r,u+1)\mathbf{C} + \bpi(n-1,r,u+1) \mathbf{D}, ~~n \ge 1,~  a \le r \le b-1, \lb{sp5}\\
\bpi(n,b,u)&=&\bpi(n,b,u+1)\mathbf{C} + \bpi(n-1,b,u+1) \mathbf{D} + \sum_{r=a}^{b} \bpi(n+b,r,1) \mathbf{C} s_b(u) \nn \\
&&+ \sum_{r=a}^{b} \bpi(n+b-1,r,1) \mathbf{D} s_b(u), ~~ n \ge 1. \lb{sp6}
\eea
Let us define the VGF of $\bpi(n,r,u)$ as
\bea \bpi^*(n,r,z)&=& \sum_{u=1}^{\infty} \bpi(n,r,u) z^u,~~ ~ |z| \leq 1, ~n \geq 0,~  a \le r \le b.
\lb{sp7} \eea
It follows from (\ref{sp7}) that
\bea \bpi^*(n,r,1) &=& \sum_{u=1}^{\infty} \bpi(n,r,u) = \bpi(n,r), ~~~~  n \geq 0,~  a \le r \le b. \nn \eea
Multiplying (\ref{sp3})-(\ref{sp6}) by $z^u$  and summing over $u$ from 1 to $\infty$, we get
\bea
\bpi^*(0,a,z)&=& \frac{1}{z} (\bpi^*(0,a,z) - z \bpi(0,a,1))\mathbf{C} + \bp(a-1,0) \mathbf{D} S^*_a(z) \nn \\
&+& \sum_{r=a}^{b} \bpi(a,r,1) \mathbf{C} S^*_a(z) + \sum_{r=a}^{b} \bpi(a-1,r,1) \mathbf{D} S^*_a(z), \lb{sp8}\\
\bpi^*(0,r,z)&=& \frac{1}{z} (\bpi^*(0,r,z) - z \bpi(0,r,1))\mathbf{C} + \sum_{i=a}^{b} \bpi(r,i,1) \mathbf{C} S^*_r(z) \nn \\
&+& \sum_{i=a}^{b} \bpi(r-1,i,1) \mathbf{D} S^*_r(z), ~~ a+1 \le r \le b, \lb{sp9}\\
\bpi^*(n,r,z)&=& \frac{1}{z} (\bpi^*(n,r,z) - z \bpi(n,r,1))\mathbf{C} \nn \\
&+& \frac{1}{z} (\bpi^*(n-1,r,z) - z \bpi(n-1,r,1)) \mathbf{D}, ~~n \ge 1,~  a \le r \le b-1, \lb{sp10}\\
\bpi^*(n,b,z)&=& \frac{1}{z} (\bpi^*(n,b,z) - z \bpi(n,b,1))\mathbf{C} + \frac{1}{z} (\bpi^*(n-1,b,z) - z \bpi(n-1,b,1)) \mathbf{D} \nn \\
&+& \sum_{r=a}^{b} \bpi(n+b,r,1) \mathbf{C} S^*_b(z) + \sum_{r=a}^{b} \bpi(n+b-1,r,1) \mathbf{D} S^*_b(z), ~~ n \ge 1. \lb{sp11}
\eea
Post multiplying by \be~ in (\ref{sp1})-(\ref{sp2}) and then summing over $n$ and $r$, we get
\bea
\bp(a-1,0) \mathbf{D} \be &=& \sum_{n=0}^{a-2} \sum_{r=a}^{b} \bpi(n,r,1) \be + \sum_{r=a}^{b} \bpi(a-1,r,1) \mathbf{C} \be. \lb{sp12}
\eea
This can be written as
\bea
\bp(a-1,0) \mathbf{D} \be &=& \sum_{n=0}^{a-1} \sum_{r=a}^{b} \bpi(n,r,1) \mathbf{C} \be + \sum_{n=0}^{a-2} \sum_{r=a}^{b} \bpi(n,r,1) \mathbf{D} \be.  \lb{sp13}
\eea
Post multiplying (\ref{sp8})-(\ref{sp11}) by \be~ and then summing over $n$ and $r$, we get
\bea \hspace{-1cm}\left(\frac{z-1}{z}\right) \sum_{n=0}^{\infty} \sum_{r=a}^{b} \bpi^*(n,r,z) \be
&=& - \sum_{n=0}^{\infty} \sum_{r=a}^{b} \bpi(n,r,1) \be \nn \\
&+& \sum_{n=0}^{a-1} \sum_{r=a}^{b} \bpi(n,r,1) \mathbf{C} \be S^*_a(z) + \sum_{n=0}^{a-2} \sum_{r=a}^{b} \bpi(n,r,1) \mathbf{D} \be S^*_a(z) \nn \\
&+& \sum_{r=a}^{b} \bpi(a,r,1) \mathbf{C} \be S^*_a(z) + \sum_{r=a}^{b} \bpi(a-1,r,1) \mathbf{D} \be S^*_a(z), \nn\\
&+& \sum_{n=a+1}^{b} \sum_{r=a}^{b} \bpi(n,r,1) \mathbf{C} \be S^*_n(z) + \sum_{n=a}^{b-1} \sum_{r=a}^{b} \bpi(n,r,1) \mathbf{D} \be S^*_n(z) \nn \\
&&\hspace{-0.8cm} +\sum_{n=b+1}^{\infty} \sum_{r=a}^{b} \bpi(n,r,1) \mathbf{C} \be S^*_b(z) +
\sum_{n=b}^{\infty} \sum_{r=a}^{b} \bpi(n,r,1) \mathbf{D} \be S^*_b(z). \lb{sp14} \eea
\bea &&\hspace{-2cm}\Rightarrow~~\left(\frac{z-1}{z}\right) \sum_{n=0}^{\infty} \sum_{r=a}^{b} \bpi^*(n,r,z) \be \nn \\
&=&  \sum_{n=0}^{a} \sum_{r=a}^{b} \bpi(n,r,1) \mathbf{C} \be (S^*_a(z)-1) + \sum_{n=0}^{a-1} \sum_{r=a}^{b} \bpi(n,r,1) \mathbf{D} \be (S^*_a(z)-1) \nn \\
&+& \sum_{n=a+1}^{b} \sum_{r=a}^{b} \bpi(n,r,1) \mathbf{C} \be (S^*_n(z)-1) + \sum_{n=a}^{b-1} \sum_{r=a}^{b} \bpi(n,r,1) \mathbf{D} \be (S^*_n(z)-1) \nn \\
&+& \sum_{n=b+1}^{\infty} \sum_{r=a}^{b} \bpi(n,r,1) \mathbf{C} \be (S^*_b(z)-1) + \sum_{n=b}^{\infty}
\sum_{r=a}^{b} \bpi(n,r,1) \mathbf{D} \be (S^*_b(z)-1). \lb{sp15} \eea
Letting $z \rightarrow 1 $ in (\ref{sp15}), we get
\bea 1 - \sum_{n=0}^{a-1} \bp(n,0) \be
&=&  \sum_{n=0}^{a} \sum_{r=a}^{b} \bpi(n,r,1) \be S_a + \sum_{n=a+1}^{b} \sum_{r=a}^{b} \bpi(n,r,1) \be S_n \nn \\
&+& \sum_{n=b+1}^{\infty} \sum_{r=a}^{b} \bpi(n,r,1)  \be S_b. \lb{sp16}
\eea
Further define
\bea
\widetilde{\bpi}^*(x,\xi,z) &=& \sum_{n=0}^{\infty} \sum_{r=a}^{b} \bpi^*(n,b,z) x^n \xi^r,~~~ |x| \le 1, ~ |\xi| \le 1, ~ |z| \le 1. \nn
\eea
Now we multiplying (\ref{sp8})-(\ref{sp11}) by $x^n$ and $\xi^r$  and summing over $n$ from 0 to $\infty$ and $r$ from $a$ to $b$, we get
\bea
\widetilde{\bpi}^*(x,\xi,z)  \left( \displaystyle{\frac{z\mathbf{I}-\mathbf{C}-\mathbf{D}x}{z}}\right)
&=& - \sum_{n=0}^{\infty} \sum_{r=a}^{b} \bpi(n,r,1) x^n \xi^r (\mathbf{C}+\mathbf{D}x) + \bp(a-1,0) \mathbf{D} \xi^a S^*_a(z) \nn \\
&+& \sum_{n=a}^{b} \sum_{r=a}^{b} (\bpi(n,r,1) \mathbf{C}   + \bpi(n-1,r,1) \mathbf{D}) \xi^n S^*_n(z) \nn \\
&+& \sum_{n=b+1}^{\infty} \sum_{r=a}^{b} (\bpi(n,r,1) \mathbf{C}   +
\bpi(n-1,r,1) \mathbf{D}) x^{n-b} \xi^b S^*_b(z). \lb{sp17} \eea
%
Our aim is to determine the bivariate VGF of the queue and server content. For this, we utilize  the
eigenvalues and eigenvectors of $(\mathbf{C+D}x)$, see Claeys et al \cite{claeys2013analysis}, Pradhan and
Gupta \cite{pradhan2017analysis}. Now let $\gamma_1(x),\gamma_2(x),\dots,\gamma_m(x)$ are the eigenvalues and
$\boldsymbol{\eta_1}(x),\boldsymbol{\eta_2}(x),\dots,\boldsymbol{\eta_m}(x)$ are the corresponding right
eigenvectors of $(\mathbf{C+D}x)$. Thus, for $1 \le i \le m$, we have
\bea
(\mathbf{C}+\mathbf{D}x)\boldsymbol{\eta_i}(x)  &=&  \gamma_i(x)\boldsymbol{\eta_i}(x), \nn \\
\{\gamma_i(x) \mathbf{I} - (\mathbf{C}+\mathbf{D}x) \} \boldsymbol{\eta_i}(x) &=& \mathbf{0}. \lb{sp18}
\eea
Setting $z=\gamma_i(x)$ in (\ref{sp17}) and post-multiplying it by $\boldsymbol{\eta_i}(x)$ on both sides and using (\ref{sp18}), we get
\bea
&&\hspace{-2cm}\sum_{n=0}^{\infty} \sum_{r=a}^{b} \bpi(n,r,1) x^n \xi^r (\mathbf{C}+\mathbf{D}x) \boldsymbol{\eta_i}(x) \nn \\
&=& \bp(a-1,0) \mathbf{D} \xi^a S^*_a(\gamma_i(x)) \boldsymbol{\eta_i}(x) \nn \\
&+& \sum_{n=a}^{b} \sum_{r=a}^{b} (\bpi(n,r,1) \mathbf{C}   + \bpi(n-1,r,1) \mathbf{D}) \xi^n S^*_n(\gamma_i(x)) \boldsymbol{\eta_i}(x) \nn \\
&+& \sum_{n=b+1}^{\infty} \sum_{r=a}^{b} (\bpi(n,r,1) \mathbf{C}   + \bpi(n-1,r,1) \mathbf{D}) x^{n-b} \xi^b
S^*_b(\gamma_i(x)) \boldsymbol{\eta_i}(x). \lb{sp19} \eea
Since (\ref{sp19}) is true for all $i$ from 1 to $m$, so we have
\bea
&& \hspace{-1cm} \sum_{n=0}^{\infty} \sum_{r=a}^{b} \bpi(n,r,1) x^n \xi^r (\mathbf{C}+\mathbf{D}x) [\boldsymbol{\eta_1}(x),\dots,\boldsymbol{\eta_m}(x)] \nn \\
&& \hspace{-0.75cm}= \bp(a-1,0) \mathbf{D} \xi^a [S^*_a(\gamma_1(x)) \boldsymbol{\eta_1}(x),\dots,S^*_a(\gamma_m(x)) \boldsymbol{\eta_m}(x)] \nn \\
&& \hspace{-0.75cm}+ \sum_{n=a}^{b} \sum_{r=a}^{b} (\bpi(n,r,1) \mathbf{C}   + \bpi(n-1,r,1) \mathbf{D})
\xi^n
[S^*_n(\gamma_1(x)) \boldsymbol{\eta_1}(x),\dots,S^*_n(\gamma_m(x)) \boldsymbol{\eta_m}(x)] \nn \\
&& \hspace{-0.75cm}+\sum_{n=b+1}^{\infty} \sum_{r=a}^{b} (\bpi(n,r,1) \mathbf{C}   + \bpi(n-1,r,1)
\mathbf{D}) x^{n-b} \xi^b [S^*_b(\gamma_1(x)) \boldsymbol{\eta_1}(x),\dots,S^*_b(\gamma_m(x))
\boldsymbol{\eta_m}(x)]. \lb{sp20} \eea
Further define \bea
\mathbf{R}(x)&=&[\boldsymbol{\eta_1}(x),\boldsymbol{\eta_2}(x),\dots,\boldsymbol{\eta_m}(x)]. \lb{sp21} \eea
The inverse of $\mathbf{R}(x)$ exists whenever each eigenvalue is of multiplicity 1, for details see
\cite{claeys2013analysis}, \cite{pradhan2017analysis}. Now using (\ref{sp21}) in (\ref{sp20}) and then
post-multiplying it by $\mathbf{R}^{-1}(x)$, we get
\bea
&& \hspace{-1cm} \sum_{n=0}^{\infty} \sum_{r=a}^{b} \bpi(n,r,1) x^n \xi^r (\mathbf{C}+\mathbf{D}x)  \nn \\
&=& \bp(a-1,0) \mathbf{D} \xi^a \mathbf{R}(x) [diag\{S^*_a(\gamma_i(x))\}_{i=1}^m] \mathbf{R}^{-1}(x) \nn \\
&+& \sum_{n=a}^{b} \sum_{r=a}^{b} (\bpi(n,r,1) \mathbf{C}   + \bpi(n-1,r,1) \mathbf{D}) \xi^n
\mathbf{R}(x) [diag\{S^*_n(\gamma_i(x))\}_{i=1}^m] \mathbf{R}^{-1}(x) \nn \\
&+& \sum_{n=b+1}^{\infty} \sum_{r=a}^{b} (\bpi(n,r,1) \mathbf{C}   + \bpi(n-1,r,1) \mathbf{D}) x^{n-b} \xi^b
\mathbf{R}(x) [diag\{S^*_b(\gamma_i(x))\}_{i=1}^m] \mathbf{R}^{-1}(x). \lb{sp22} \eea
where $[diag\{S^*_r(\gamma_i(x))\}_{i=1}^m],~a\le r \le b$, is a diagonal matrix of order $m$ with diagonal
entries $S^*_r(\gamma_1(x)),\dots,S^*_r(\gamma_m(x))$. Using the theory of eigenvalues and eigenvectors, we
can write
\bea
\mathbf{C}+\mathbf{D}x&=&\mathbf{R}(x) [diag\{\gamma_i(x)\}_{i=1}^m] \mathbf{R}^{-1}(x), \nn \\
\hbox{and~} \hspace{2cm} S^*_r(\mathbf{C}+\mathbf{D}x) &=& \mathbf{R}(x) [diag\{S^*_r(\gamma_i(x))\}_{i=1}^m]
\mathbf{R}^{-1}(x). \hspace{4cm} \lb{sp23} \eea
Now using (\ref{FA(s)}) and (\ref{sp23}) in (\ref{sp22}), we obtain
\bea \hspace{-1cm} \sum_{n=0}^{\infty} \sum_{r=a}^{b} \bpi(n,r,1) x^n \xi^r (\mathbf{C}+\mathbf{D}x)
&=& \bp(a-1,0) \mathbf{D} \xi^a \mathbf{A}^*_a(x) \nn \\
&+& \sum_{n=a}^{b} \sum_{r=a}^{b} (\bpi(n,r,1) \mathbf{C}   + \bpi(n-1,r,1) \mathbf{D}) \xi^n
\mathbf{A}^*_n(x) \nn \\
&+& \sum_{n=b+1}^{\infty} \sum_{r=a}^{b} (\bpi(n,r,1) \mathbf{C}   + \bpi(n-1,r,1) \mathbf{D}) x^{n-b} \xi^b
\mathbf{A}^*_b(x), \lb{sp24} \eea
where $\mathbf{A}^*_r(x)=S^*_r(\mathbf{C}+\mathbf{D}x),~(a \le r \le b)$. Now using (\ref{sp24}), we obtain
the bivariate VGF of the queue-content at the departure epoch which is given in the next section.
\subsection{Bivariate VGF at departure epoch}\label{sec3.1}
Let us define $\bpi^+(n,r) = [\pi_{1}^+(n,r),\pi_{2}^+(n,r),\dots,\pi_{m}^+(n,r)]$ as the joint probability
vector whose $j$th element $(\pi_{j}^+(n,r))$ is the probability that there are $n$ number of customers in
the queue at departure epoch of a batch of size $r$ and arrival process is in phase $j$.\\
$\bphi^+(n)$ = Probability vector that there are $n$ number of customers in the queue at departure epoch of a batch
 = $\sum\limits_{r=a}^{b}\bpi^+(n,r)$. Let us define
$\widetilde{\bpi}^+(x,\xi)=\sum\limits_{n=0}^{\infty} \sum\limits_{r=a}^{b} \bpi^+(n,r) x^n \xi^r,$ and
$\bPhi^+(x)=\sum\limits_{n=0}^{\infty}\bphi^+(n)x^n$. Using probabilistic argument, $\bpi^+(n,r)$ and
$\bpi(n,r,1)$ are connected as: \bea
\bpi^+(0,r) &=&  \Omega \bpi(0,r,1)\mathbf{C}, \lb{dsp1}\\
\bpi^+(n,r) &=&  \Omega (\bpi(n,r,1)\mathbf{C} + \bpi(n-1,r,1)\mathbf{D}),~ n \ge 1, ~ a \le r \le b.   \lb{dsp2}
\eea
where $\Omega^{-1}= \sum\limits_{n=0}^{\infty}  \sum\limits_{r=a}^{b} \bpi(n,r,1)\be$.
\begin{lemma} \label{splemma1}
\bea
\Omega^{-1} &=& \psi^{-1} \left( 1- \sum_{n=0}^{a-1} \bp(n,0) \be\right), \nn
\eea
where $\psi=\sum\limits_{n=0}^{a} \bphi^+(n) \be S_a + \sum\limits_{n=a+1}^{b} \bphi^+(n) \be S_n
+ \sum\limits_{n=b+1}^{\infty} \bphi^+(n)  \be S_b$.
\end{lemma}
\begin{proof}
Equation (\ref{sp16}) can be written as
\bea
1 - \sum_{n=0}^{a-1} \bp(n,0) \be &=&  \sum_{r=a}^{b} \bpi(0,r,1) \mathbf{C} \be S_a \nn\\
&+& \sum_{n=1}^{a} \sum_{r=a}^{b} (\bpi(n,r,1)\mathbf{C}+\bpi(n-1,r,1)\mathbf{D}) \be S_a \nn \\
&+& \sum_{n=a+1}^{b} \sum_{r=a}^{b} (\bpi(n,r,1)\mathbf{C}+\bpi(n-1,r,1)\mathbf{D}) \be S_n \nn \\
&+& \sum_{n=b+1}^{\infty} \sum_{r=a}^{b} (\bpi(n,r,1)\mathbf{C}+\bpi(n-1,r,1)\mathbf{D})  \be S_b. \nn
\eea
Multiplying by $\Omega$ in above equation, we get
\bea
\Omega \left(1 - \sum_{n=0}^{a-1} \bp(n,0) \be \right) &=& \sum_{n=0}^{a} \bphi^+(n) \be S_a + \sum_{n=a+1}^{b}
 \bphi^+(n) \be S_n + \sum_{n=b+1}^{\infty} \bphi^+(n)  \be S_b. \nn
\eea
Hence proved.
\end{proof}
\begin{lemma} \label{splemma2}
\bea
\Omega \bp(a-1,0) &=& \sum_{n=0}^{a-1} \bphi^+(n) (\overline{\mathbf{D}})^{(a-1-n)} (\mathbf{I}-\mathbf{C})^{-1}, \nn
\eea
where $\overline{\mathbf{D}}=(\mathbf{I}-\mathbf{C})^{-1}\mathbf{D}$.
\end{lemma}
\begin{proof}
Multiplying by $\Omega$ in (\ref{sp1}), we get
\bea
\Omega \bp(0,0) (\mathbf{I}-\mathbf{C}) &=& \bphi^+(0), \nn \\
\Omega \bp(0,0)  &=& \bphi^+(0)(\mathbf{I}-\mathbf{C})^{-1}, \nn
\eea
Now multiplying by $\Omega$ in (\ref{sp2}) for $n=i$ and after simplification, finally we get
\bea
\Omega \bp(i,0) (\mathbf{I}-\mathbf{C}) &=& \sum_{n=0}^{i} \bphi^+(n) (\overline{\mathbf{D}})^{(i-n)},~~  1 \le i \le a-1, \nn
\eea
which gives
\bea
\Omega \bp(a-1,0) (\mathbf{I}-\mathbf{C}) &=& \sum_{n=0}^{a-1} \bphi^+(n) (\overline{\mathbf{D}})^{(a-1-n)}, \nn \\
\Omega \bp(a-1,0)  &=& \sum_{n=0}^{a-1} \bphi^+(n) (\overline{\mathbf{D}})^{(a-1-n)}
(\mathbf{I}-\mathbf{C})^{-1}. \nn \eea
Hence proved.
\end{proof}
Now multiplying (\ref{sp24}) by $\Omega$ and using the definition of departure epoch probabilities and Lemma
\ref{splemma1}, we get
\bea \widetilde{\bpi}^+(x,\xi) = \sum_{n=0}^{a-1} \bphi^+(n) (\overline{\mathbf{D}})^{(a-n)}
 \xi^a \mathbf{A}^*_a(x) + \sum_{n=a}^{b}  \bphi^+(n) \xi^n \mathbf{A}^*_n(x) + \frac{\xi^b}{x^b}
 \sum_{n=b+1}^{\infty} \bphi^+(n) x^{n} \mathbf{A}^*_b(x).  \lb{dsp3}
\eea
Setting $\xi=1$ in (\ref{dsp3}), we get the VGF of $only$ queue length distribution
$\bPhi^+(x)(=\widetilde{\bpi}^+(x,1))$ as
\bea \bPhi^+(x)=\sum_{n=0}^{a-1} \bphi^+(n) (\overline{\mathbf{D}})^{(a-n)} \mathbf{A}^*_a(x) +
\sum_{n=a}^{b}  \bphi^+(n)  \mathbf{A}^*_n(x) + \frac{1}{x^b} \sum_{n=b+1}^{\infty} \bphi^+(n) x^{n}
\mathbf{A}^*_b(x), \lb{dsp4} \eea
which gives
\bea
\bPhi^+(x) (x^b \mathbf{I}-\mathbf{A}^*_b(x)) &=& \sum_{n=0}^{a-1} \bphi^+(n) \Big( (\overline{\mathbf{D}})^{(a-n)} x^b
\mathbf{A}^*_a(x) - x^{n} \mathbf{A}^*_b(x) \Big) \nn \\
&+&  \sum_{n=a}^{b-1}  \bphi^+(n) \Big( x^b \mathbf{A}^*_n(x) - x^{n} \mathbf{A}^*_b(x) \Big).  \lb{dsp6}
\eea
From (\ref{dsp3}), we have
\bea
\widetilde{\bpi}^+(x,\xi) &=& \sum_{n=0}^{a-1} \bphi^+(n) (\overline{\mathbf{D}})^{(a-n)}
 \xi^a \mathbf{A}^*_a(x) + \sum_{n=a}^{b}  \bphi^+(n) \xi^n \mathbf{A}^*_n(x) \nn \\
 &+& \frac{\xi^b}{x^b} \left(\bPhi^+(x) - \sum_{n=0}^{b} \bphi^+(n) x^{n} \right) \mathbf{A}^*_b(x).  \lb{dsp7}
\eea
Now substituting the value of the vector $\bPhi^+(x)$ from (\ref{dsp6}) to (\ref{dsp7}), we get
\bea
\widetilde{\bpi}^+(x,\xi) &=& \sum_{n=0}^{a-1} \bphi^+(n) (\overline{\mathbf{D}})^{(a-n)}
 \xi^a \mathbf{A}^*_a(x) + \sum_{n=a}^{b}  \bphi^+(n) \xi^n \mathbf{A}^*_n(x) \nn \\
 &+& \frac{\xi^b}{x^b} \Big( \Big[ \sum_{n=0}^{a-1} \bphi^+(n) \Big( (\overline{\mathbf{D}})^{(a-n)} x^b
\mathbf{A}^*_a(x) - x^{n} \mathbf{A}^*_b(x) \Big) \nn \\
&+&  \sum_{n=a}^{b-1}  \bphi^+(n) \Big( x^b \mathbf{A}^*_n(x) - x^{n} \mathbf{A}^*_b(x) \Big) \Big] (x^b \mathbf{I}-\mathbf{A}^*_b(x))^{-1} \nn \\
&-& \sum_{n=0}^{b} \bphi^+(n) x^{n} \Big) \mathbf{A}^*_b(x).  \lb{dsp8}
\eea
Post multiplying by $(x^b \mathbf{I}-\mathbf{A}^*_b(x))$ on both sides of (\ref{dsp8}), we obtain
\bea
\widetilde{\bpi}^+(x,\xi)
&=& \Bigg( \sum_{n=0}^{a-1} \bphi^+(n) \Big( (\xi^b-\xi^a) (\overline{\mathbf{D}})^{(a-n)} \mathbf{A}^*_a(x) \mathbf{A}^*_b(x) + (\overline{\mathbf{D}})^{(a-n)} \xi^a x^b  \mathbf{A}^*_a(x) - \xi^b x^n  \mathbf{A}^*_b(x) \Big) \nn \\
&+&  \sum_{n=a}^{b-1}  \bphi^+(n) \Big( (\xi^b-\xi^n) \mathbf{A}^*_n(x)\mathbf{A}^*_b(x) + \xi^n x^b \mathbf{A}^*_n(x) - \xi^b x^{n} \mathbf{A}^*_b(x) \Big)\Bigg)(x^b \mathbf{I}-\mathbf{A}^*_b(x))^{-1}. \nn \\
 \lb{dsp9}
\eea
\section{Joint queue and server content distributions at various epochs} \label{sec4}
In this section, we obtain joint queue and server content distribution at arbitrary, pre-arrival and outside
observer's epochs.
\subsection{Joint queue and server content distribution at arbitrary epoch}
The joint distribution of queue and server content at arbitrary epoch plays an important role in obtaining
system length distribution and also in evaluation of several key performance measures of the queueing model
under consideration. The following theorem presents a correspondence between departure and arbitrary epoch
probability vectors.
\begin{theorem}
The steady-state probability vectors $\{\bp(n,0),\bpi(n,r)\}$ and $\{\bpi^+(n,r),\bphi^+(n)\}$ are connected
by
\bea
\bp(n,0)  &=& \Omega^{-1} \sum_{j=0}^{n} \bphi^+(j) (\overline{\mathbf{D}})^{(n-j)} (\mathbf{I}-\mathbf{C})^{-1},~~~ 0 \le n \le a-1, \lb{asp1} \\
\bpi(0,a) &=& \Big( \bp(a-1,0) \mathbf{D} + \Omega^{-1} (\bphi^+(a)-\bpi^+(0,a)) \Big)(\mathbf{I}-\mathbf{C})^{-1}, \lb{asp2}  \\
\bpi(0,r) &=& \Omega^{-1} (\bphi^+(r)-\bpi^+(0,r))(\mathbf{I}-\mathbf{C})^{-1}, ~~~~ a+1 \le r \le b, \lb{asp3} \\
\bpi(n,r) &=& \Big( \bpi(n-1,r) \mathbf{D} - \Omega^{-1} \bpi^+(n,r)\Big)(\mathbf{I}-\mathbf{C})^{-1},~~n \ge 1,~~ a \le r \le b-1, \lb{asp4}\\
\bpi(n,b) &=& \Big( \bpi(n-1,b) \mathbf{D} + \Omega^{-1} (\bphi^+(n+b)-\bpi^+(n,b)) \Big)(\mathbf{I}-\mathbf{C})^{-1},~~n \ge 1. \lb{asp5}
\eea
\end{theorem}
\begin{proof}
From Lemma \ref{splemma2}, we obtain (\ref{asp1}). Now setting $z=1$ in (\ref{sp8})-(\ref{sp11}), we get
\bea
\bpi(0,a) (\mathbf{I}-\mathbf{C})&=& -  \bpi(0,a,1) \mathbf{C} + \bp(a-1,0) \mathbf{D} + \sum_{r=a}^{b} \bpi(a,r,1) \mathbf{C} + \sum_{r=a}^{b} \bpi(a-1,r,1) \mathbf{D}, \lb{asp8}\\
\bpi(0,r) (\mathbf{I}-\mathbf{C})&=& -  \bpi(0,r,1) \mathbf{C} + \sum_{i=a}^{b} \bpi(r,i,1) \mathbf{C}  + \sum_{i=a}^{b} \bpi(r-1,i,1) \mathbf{D}, ~~ a+1 \le r \le b, \lb{asp9}\\
\bpi(n,r) (\mathbf{I}-\mathbf{C})&=& -  \bpi(n,r,1)\mathbf{C} + (\bpi(n-1,r) -  \bpi(n-1,r,1)) \mathbf{D}, ~n \ge 1,~  a \le r \le b-1,~~ \lb{asp10}\\
\bpi(n,b) (\mathbf{I}-\mathbf{C})&=& -  \bpi(n,b,1)\mathbf{C} + (\bpi(n-1,b) -  \bpi(n-1,b,1)) \mathbf{D} + \sum_{r=a}^{b} \bpi(n+b,r,1) \mathbf{C}  \nn \\
&+& \sum_{r=a}^{b} \bpi(n+b-1,r,1) \mathbf{D}, ~ n \ge 1. \lb{asp11}
\eea
Now multiplying (\ref{asp8})-(\ref{asp11}) by $\Omega$ and using Lemma \ref{splemma1},\ref{splemma2} and then using the definition of departure epoch, we get the required relation between arbitrary and departure epoch as given in (\ref{asp2})-(\ref{asp5}).
\end{proof}
\subsection{Queue length and server content distribution at pre-arrival epoch}
Let $\bp^-(n,0),~(0 \le n \le a-1)$ and $\bpi^-(n,r),~(a \le r \le b, n \ge 0)$  be the $1 \times m$ vectors with $i$th  component as
$p_{i}^-(n,0)$ and $\pi_{i}^-(n,r)$, respectively. Let us define $p_{i}^-(n,0)$ as the steady-state probability that an arrival find $n,~ (0 \le n \le a -1)$ customers in the queue, server idle, and phase of the arrival process is $i$. Similarly, we define $\pi_{i}^-(n,r)$ to be the steady-state probability that an arrival finds $n (\ge 0)$ customers in the queue, server busy with $r,~ (a \le r \le b)$ customers and phase of the arrival process is $i$. It can be easily shown (see \cite{pradhan2017analysis}) that the vectors $\bp^-(n,0)$ and $\bpi^-(n,r)$ are given by
\bea
\bp^-(n,0) &=&  \frac{\bp(n,0)\mathbf{D}}{\lambda^{*}},~~~~ 0 \le n \le a-1, \lb{asp12} \\
\bpi^-(n,r) &=&  \frac{\bpi(n,r)\mathbf{D}}{\lambda^{*}},~~~~ a \le r \le b, n \ge 0. \lb{asp13}
\eea
\subsection{Queue length and server content distribution at outside observer's
epoch}\label{sec4.1}
In LAS-DA, since an outside observer's observation epoch falls in a
time interval after the potential departure of a batch and before a
potential arrival, the probability vector $\bpi^o(n,r), ~ (a \le r
\le b, n \ge 0)$ that an outside observer sees there are $n$ number
of customers in the queue and $r$ with the server is the same as
that of the arbitrary epoch $\bpi(n,r),  ~ (a \le r \le b, n \ge
0)$, $i.e.$ $ \bpi^o(n,r)=\bpi(n,r), ~ (a \le r \le b, n \ge 0).$
\par
This completes the theoretical analysis of the model under consideration. In the next section we present a
step-wise procedure for computing the distribution at various epochs. One can observe from Section \ref{sec4}
that in order to obtain the distributions at various epochs, first we have to find the distributions at
departure epoch which in given in the next section.
\section{Computing process to obtain the distributions at various
epochs}\label{sec5}
In this section we present the step-wise computing procedure for
evaluation of the distribution at departure epoch. In order to
extract the probability distribution from (\ref{dsp9}) first we have
to determine the unknown probability vectors
$\{\bphi^+(n)\}_{n=0}^{b-1}$. So in total we have to determine $mb$
unknowns $i.e.$ $\{\phi^+_i(n)\}_{n=0}^{b-1},~ 1 \le i \le m$, we
obtain these unknowns from (\ref{dsp6}) using the roots method given
in Chaudhry et al \cite{chaudhry2013simple}, Gupta et al
\cite{gupta2016alternative}, Pradhan and Gupta
\cite{pradhan2017analysis}. For this, first we obtain the
expressions of $\mathbf{A}^*_r(x), ~ a \le r \le b$, by considering
commonly used service-time distributions.
\subsection{Evaluation of $\mathbf{A}^*_r(x), ~ a \le r \le b$}
In this section, we obtain the expression of $\mathbf{A}^*_r(x), ~ a \le r \le b$ when service-time
distribution follows: (i) discrete phase-type ($DPH$) distribution (ii) arbitrary distribution. These
distributions cover almost all types of distributions that arise in many real life situations.
\subsubsection{Service-time follows $DPH$ distribution}
Let service-time follows a $DPH$ distribution with representation $DPH_r(\boldsymbol{\beta}_r,\mathbf{T}_r),
(a \le r \le b)$, where $\boldsymbol{\beta}_r$,  $\mathbf{T}_r$ are row vector and matrix, respectively, of
dimension $\nu$. We have $s_r(k)=\boldsymbol{\beta}_r\mathbf{T}_r^{k-1}\mathbf{T}_r^0$, where $\mathbf{T}_r^0
=(\mathbf{I}_{\nu}-\mathbf{T}_r)\mathbf{e}$. Using (\ref{sp23}), we obtain
\bea \mathbf{A}^*_r(x)
&=& \sum_{k=1}^{\infty} [\mathbf{D}(x)]^k \otimes s_r(k), \nn \\
&=& \sum_{k=1}^{\infty} [\mathbf{I}_m (\mathbf{D}(x))^{k-1} \mathbf{D}(x)]  \otimes (\boldsymbol{\beta}_r\mathbf{T}_r^{k-1}\mathbf{T}_r^0) \nn \\
&=& (\mathbf{I}_m \otimes \boldsymbol{\beta}_r) \Big(\sum_{k=1}^{\infty}   (\mathbf{D}(x))^{k-1} \otimes
\mathbf{T}_r^{k-1} \Big)
(\mathbf{D}(x) \otimes \mathbf{T}_r^0) \nn \\
&=& (\mathbf{I}_m \otimes \boldsymbol{\beta}_r) \Big(\sum_{k=1}^{\infty}   (\mathbf{D}(x) \otimes
\mathbf{T}_r)^{k-1}\Big)
(\mathbf{D}(x) \otimes \mathbf{T}_r^0). \nn \\
\mathbf{A}^*_r(x)&=& (\mathbf{I}_m \otimes \boldsymbol{\beta}_r) ( \mathbf{I}_{m \nu} - \mathbf{D}(x) \otimes
\mathbf{T}_r)^{-1} (\mathbf{D}(x) \otimes \mathbf{T}_r^0), \nn \eea
where $\otimes$ is used for the Kronecker product. \\
Since the inverse term is appearing in the expression of $\mathbf{A}^*_r(x)$, we can write $\mathbf{A}^*_r(x)$ as $\mathbf{A}^*_r(x)=\frac{\mathbf{X}_r(x)}{y_r(x)}$, where $y_r(x)$ in the determinant of the corresponding term. So we can conclude that each element of the matrix $\mathbf{A}^*_r(x)$ is a rational function with the denominator as $y_r(x)$.
\begin{remark}
(i) If we set  $\boldsymbol{\beta}_r=(1)$ and $\mathbf{T}_r=[1-\mu_r]$, we get $\mathbf{A}^*_r(x)$ for geometric service time distribution.\\
(ii) If we set  $\boldsymbol{\beta}_r=(1,0,0,\dots,0)$ and $\mathbf{T}_r=\begin{bmatrix}
    1-\mu_r            & \mu_r          &        &           &       \\
                     & 1-\mu_r        & \mu_r    &           &       \\
                     &              &  .     &   .       &       \\
                     &              &        &     .     &   .    \\
                     &              &        &           & 1-\mu_r
\end{bmatrix}$,
we get $\mathbf{A}^*_r(x)$ for negative binomial service time distribution.
\end{remark}
\subsubsection{Service-time follows arbitrary distribution}
Let service-time is arbitrarily distributed with maximum $K$ slots
so that \bea S^*_r (z)         &=& \sum\limits_{n=1}^{K} s_r(n)
z^n,~~~ \sum_{i=1}^{K} s_r(i)=1. \nn \eea This leads to \bea
\mathbf{A}^*_r(x) &=& \sum\limits_{n=1}^{K} s_r(n)
(\mathbf{D}(x))^n. \nn \eea
\begin{remark}
If $s_r(K)=1$ and $s_r(i)=0,$ (for all $ i < K$) then we obtain $\mathbf{A}^*_r(x)$ for deterministic service time distribution with parameter $K$ $ i.e.$ $\mathbf{A}^*_r(x)=(\mathbf{D}(x))^K$.
\end{remark}
\begin{remark}
From the above expressions of $\mathbf{A}^*_r(x)$, we can conclude that each element of $\mathbf{A}^*_r(x)$ can be written as $\mathbf{A}^*_r(x)=\frac{\mathbf{X}_r(x)}{y_r(x)}$ (for arbitrary service time distribution $y_r(x)=1$).
\end{remark}
\subsection{Computing process for evaluation of distributions at departure epochs}
First we present step-wise computing process for the evaluation of
unknown probability vector, then using it we extract the remaining
probability vectors.
\subsubsection{Determination of unknown probability vectors}
As each element of the matrix $\mathbf{A}^*_r(x)$ is a rational
function, so we assume that the $(i,j)$-th element of
$\mathbf{A}^*_r(x)$ is $\frac{X_{r;i,j}(x)}{y_r(x)}$. So we have
$(i,j)$-th element of $(x^b \mathbf{I}-\mathbf{A}^*_b(x))$ as \bea
[x^b \mathbf{I}-\mathbf{A}^*_b(x)]_{i,j} &=&
\frac{w_{i,j}(x)}{y_b(x)},~\hbox{where}~ w_{i,j}(x)~=~\left\{
                \begin{array}{ll}
                  x^by_b(x)-X_{b;i,j}(x),  &  i=j \\
                  -X_{b;i,j}(x),           &  i \ne j.
                \end{array}
              \right.    \lb{cdsp1}
\eea
Now (\ref{dsp6}) can be written in the following $m$ simultaneous equations in $m$ unknowns $\boldsymbol{\phi}_j^+(x),~~1\leq j\leq m$.
\bea
w_{1,1}(x)\phi_1^+(x)+w_{2,1}(x)\phi_2^+(x)+\ldots+w_{m,1}(x)\phi_m^+(x)&=&\Theta_1(x)\nn\\
w_{1,2}(x)\phi_1^+(x)+w_{2,2}(x)\phi_2^+(x)+\ldots+w_{m,2}(x)\phi_m^+(x)&=&\Theta_2(x)\nn\\
\vdots&& \vdots\nn\\
w_{1,m}(x)\phi_1^+(x)+w_{2,m}(x)\phi_2^+(x)+\ldots+w_{m,m}(x)\phi_m^+(x)&=&\Theta_m(x),\nn
\eea
where $\Theta_j(x),~1 \le j \le m$ is given as
\bea
\Theta_j(x)&=&\left[\prod_{i=a+1}^{b-1}y_{i}(x)\left\{ \sum_{l=1}^{m} \sum_{n=0}^{a-1}\sum_{k=1}^{m}\phi^+_{k}(n)  (\overline{D})^{(a-n)}_{k,l} X_{a;l,j}(x)x^b y_{b}(x)\right.\right.\nn\\
&&-\left.\left.\sum_{l=1}^{m}\sum_{n=0}^{a-1}\phi^+_{l}(n)x^n X_{b;l,j}(x)y_{a}(x) \right\} -\sum_{l=1}^{m}\sum_{n=a}^{b-1}\phi^+_{l}(n) \left\{ x^b y_{b}(x) X_{n;l,j}(x)\right.\right.\nn\\
&&\left.\left.-x^n y_{n}(x) X_{b;l,j}(x)\right\} \prod_{i=a,~ i\neq n}^{b-1}y_{i}(x)\right] \big/
\left[\prod_{i=a}^{b-1}y_{i}(x)\right]. \lb{cdsp2} \eea
Now solving the above system of equations using Cramer's rule, we obtain $\phi_j^+(x),~1\leq j \leq m$, as
\bea
\phi_j^+(x)=\frac{\left|V_j(x)\right|}{\left|V(x)\right|},~~1\leq j \leq m, \lb{cdsp3}
\eea
where both $V_j(x)$ and $V(x)$ represent square matrix with $(k,l)$-th elements given by
\bea
[V_j(x)]_{k,l} = \left\{
                   \begin{array}{ll}
                     v_{l,k}(x), & l \neq j \\
                     \Theta_{k}(x), & l= j
                   \end{array}
                 \right.
~~~~ \hbox{and}~~~~ [V(x)]_{k,l}=v_{l,k}(z). \lb{cdsp4} \eea
The
$j$-th column of the square matrix $V_j(x)$ is replaced by
$[\Theta_1(x),\Theta_2(x),\ldots,\Theta_m(x)]^T$ and all other
elements are the same as those of $V(x)$.
\par
Let us assume that $|V(x)|$ is a polynomial in $x$ must possess a non-zero coefficient of power of $x$.
Finally, we have \bea \phi_j^+(x)=\frac{\Upsilon_j(x)}{\Upsilon(x)},~~1\leq j \leq m, \lb{cdsp5} \eea where
$\Upsilon_j(x)=|V_j(x)|$ and $\Upsilon(x)=|V(x)|$. To be more specific what we are having now is the pgf of
\emph{only} queue length distribution for each phase at departure epoch. Till now we have not discussed about
the determination of unknown probability vectors. To do this we consider (\ref{cdsp5}), and let us call
$\Upsilon(x) = 0$ as characteristic equation associated with the pgf of each phase. It can be easily shown
that $|x^b \mathbf{I}-\mathbf{A}^*_b(x)|\equiv \frac{\Upsilon(x)}{\{y_{b}(x)\}^m}=0$ has exactly $mb$ roots
inside and on the closed complex unit disk $|x|\leq 1$, see Gail et al. \cite{gail1995linear} [p. 5]. Let us
assume that these roots are distinct and denote them as $x_1,x_2,\ldots,x_{mb}$ with $x_{mb}=1$. However in
case of repeated roots the procedure has to be modified slightly which is standard in the literature of
queueing.
\par
Analyticity of $\phi_j^+(x)$ in $|x|\leq 1$ implies that the roots $x_1,x_2,\ldots,x_{mb-1}$ of $\Upsilon(x)=0$ (the denominator of (\ref{cdsp5})) must coincide with that of numerator. Thus by taking any one component of $\bPhi^+(x)$, say $\phi_j^+(x),~(1\leq j \leq m)$ we are led to $mb-1$ equations as
\bea
\Upsilon_j(x_i)=0,~~~1\leq i \leq mb-1. \lb{cdsp6}
\eea
The necessity of one more equation can be fulfilled by employing the normalizing condition $\bPhi^+(1)e=1$ which leads to
\bea
\sum_{j=1}^{m}\Upsilon_j^{'}(1)= \Upsilon^{'}(1). \lb{cdsp7}
\eea
Solving (\ref{cdsp6}) and (\ref{cdsp7}) together, we get the $mb$ unknowns
$\phi^+_j(n),~~(0\leq n\leq b-1,~1\leq j \leq m)$.
\subsection{Extraction of probability vectors from bivariate VGF}
In the previous section we have obtained unknown probability vectors $\{\bphi^+(n)\}_{n=0}^{b-1}$. So we have
bivariate VGF $\widetilde{\bpi}^+(x,\xi)$ in completely known form.  Now our aim is to extract the
probability vectors $\bpi^+(n,r),~n\ge 0,~a\le r \le b$, which can be done by inverting
$\widetilde{\bpi}^+(x,\xi)$, which is not easily tractable. To make it simpler we first collect the
coefficient of $\xi^j,~a\le j\le b$, from both the sides of (\ref{dsp9}) that are given by
\bea \mbox{coefficient of} ~ \xi^a:~~~~ \sum_{n=0}^{\infty}\bpi^+(n,a)x^n
&=& \sum_{n=0}^{a}\bphi^+(n)\overline{\mathbf{D}}^{a-n}\mathbf{A}^*_{a}(x). \lb{cdsp8}\\
\mbox{coefficient of} ~ \xi^j:~~~~\sum_{n=0}^{\infty}\bpi^+(n,j)x^n
&=& \bphi^+(j)\mathbf{A}^*_{j}(x),~~ a+1\le j \le b-1.
\lb{cdsp9}
\eea
\bea~~~~\mbox{coefficient of} ~ \xi^b:~~~~\nn\\
\sum_{n=0}^{\infty}\bpi^+(n,b)x^n &=& \left[\sum_{n=0}^{a-1}\bphi^+(n)\left\{\overline{\mathbf{D}}^{a-n}\mathbf{A}^*_{a}(x)-
x^n \mathbf{I}\right\}\right.\nn\\
&&\left.+\sum_{n=a}^{b-1}\bphi^+(n)\left\{\mathbf{A}^*_{n}(x)-x^n \mathbf{I} \right\}\right]
\mathbf{A}^*_{b}(x) \left[x^b \mathbf{I}- \mathbf{A}^*_{b}(x)\right]^{-1}. \lb{cdsp10} \eea
Now collecting the coefficient of $x^n$ from both the sides of (\ref{cdsp8}) and (\ref{cdsp9}) we get \bea
\bpi^+(n,a) &=& \sum_{n=0}^{a}\bphi^+(n) \overline{\mathbf{D}}^{a-n}\mathbf{A}_{a}(n),~~n\geq 0. \lb{cdsp11}\\
\bpi^+(n,j) &=& \bphi^+(j)\mathbf{A}_{j}(n),~~a+1\le j \le b-1,~~n\geq 0. \lb{cdsp12}
\eea
%
%
Now only $\bpi^+(n,b)$ is left to be determined. That can be done by
inverting (\ref{cdsp10}), where each component of the vector is a
polynomial in $x$ for a specific service time distribution.
\par
Let us denote $\sum_{n=0}^{\infty}\bpi^+(n,b)x^n$ as
$\boldsymbol{\psi}^+(x)=\left[\psi_1^+(x),\psi_2^+(x),\ldots,\psi_m^+(x)\right]$ to make the analysis easier
for the rest portion of this section. In order to extract the probability vectors from
$\boldsymbol{\psi}^+(x)$ the same analysis as in the previous section for $\bPhi^+(x)$ has to be carried out.
In view of this, $\bPhi^+(x)$ and $\Theta_j(x)$ (used in earlier case in eqn. (\ref{cdsp2})) has to be
replaced by $\boldsymbol{\psi}^+(x)$ and $F_j(x)$, respectively, where $F_j(x)$, ($1\leq j \leq m$) is given
by
\bea F_j(x) &=& \left[\sum_{i=1}^{m}\left\{\sum_{l=1}^{m}\left(\sum_{n=0}^{a-1}\sum_{k=1}^{m}\phi^+_{k}(n)
(\overline{D})^{(a-n)}_{k,l}
\prod_{v=a+1}^{b-1} y_{v}(x) X_{a;l,i}(x) \right.\right.\right.\nn\\
&&~~~~~~~~~~~~~~\left.\left.\left.+\sum_{n=a}^{b-1}\phi^+_{l}(n) \prod_{v=a;v\ne n}^{b-1} y_{v}(x)
X_{n;l,i}(x) \right)\right.\right.\nn\\
&&\left.\left.-\left(\sum_{n=0}^{b-1}\phi^+_{i}(n)x^n \right)\prod_{v=a}^{b-1} y_{v}(x)\right\} X_{b;i,j}(x)
\right] \big/\left[\prod_{v=a}^{b-1} y_{v}(x)\right]. \lb{cdsp13} \eea
Therefore the simplified form of $\psi^+_j(x)$ is given by
\bea
\psi_j^+(x) = \frac{\left|V_j(x)\right|}{\left|V(x)\right|},~~1\le j \le m, \lb{cdsp14}
\eea
where both $V_j(x)$ and $V(x)$ represent square matrix with $(k,l)$-th elements given by
\bea
[V_j(x)]_{k,l} = \left\{
                   \begin{array}{ll}
                     v_{l,k}(x), & l \neq j \\
                     F_{k}(x), & l= j
                   \end{array}
                 \right.
~~~~ \hbox{and}~~~~ [V(x)]_{k,l}=v_{l,k}(z). \lb{cdsp15} \eea
The
$j$-th column of the square matrix $V_j(x)$ is replaced by
$[F_1(x),F_2(x),\ldots,F_m(x)]^T$ and all other elements are the
same as those of $V(x)$.
\par
Let us assume that $|V(x)|$ is a polynomial in $x$ must possess a non-zero coefficient of power of $x$.
Finally, we have \bea \psi_j^+(x)=\frac{\Upsilon_j(x)}{\Upsilon(x)},~~1\leq j \leq m, \lb{cdsp16} \eea where
$\Upsilon_j(x)=|V_j(x)|$ and $\Upsilon(x)=|V(x)|$. Now as $\psi_j^+(x)$ is a rational function in completely
known polynomials, we can proceed to find its partial fraction. Let us assume that $\Upsilon_j(x)$ and
$\Upsilon(x)$ are the polynomials of degree $L_1$ and $M_1$, respectively.
\par
We already know that $\Upsilon(x)=0$ has $mb$ roots inside or on the unit circle. So there are total
$(M_1-mb)$ distinct roots of $\Upsilon(x)=0$ in $|x|>1$ (for repeated roots see Remark \ref{spRemark1}).  Let
us denote these roots by $\alpha_1,\alpha_2,\dots,\alpha_{M_1-mb}$. Now based on the value of $L_1$ and $M_1$
following two cases arise:
\begin{description}
  \item[Case-1:] $L_1 \ge M_1$ \\
  Applying the partial-fraction expansion, the rational function $\psi_j^+(x)~(1\le j \le m)$, can be uniquely written as
\bea \psi_j^+(x)=\sum_{i=0}^{L_1-M_1}\tau_{i,j}x^i +\sum_{k=1}^{M_1-mb}\frac{\gamma_{k,j}}{\alpha_{k}-x},
\lb{cdsp17} \eea for some constants $\tau_{i,j}$ and $\gamma_{k,j}$'s. The first sum is the result of the
division of the polynomial $\Upsilon_j(x)$ by $\Upsilon(x)$ and the constants $\tau_{i,j}$ are the
coefficients of the resulting quotient. Using the residue theorem, we have \bea
\gamma_{k,j}=-\frac{\Upsilon_j(\alpha_{k})}{\Upsilon^{\prime}(\alpha_{k})},\quad k=1,2,\ldots,M_1-mb. \nn
\eea Now, collecting the coefficient of $x^n$ from both the sides of (\ref{cdsp17}), we have \bea
\pi_j^+(n,b) &=& \tau_{n,j} + \sum\limits_{k=1}^{M_1-mb}\frac{\gamma_{k,j}}{\alpha^{n+1}_{k}},\quad n\geq 0.
\lb{cdsp18} \eea
\item[Case-2:] $L_1 < M_1$\\
Using partial-fraction technique on $\psi_j^+(x)$ we have
\bea
\psi_j^+(x)=\sum_{k=1}^{M_1-mb}\frac{\gamma_{k,j}}{\alpha_{k}-x}, \lb{cdsp19}
\eea
where
\bee
\gamma_{k,j}=-\frac{\Upsilon_j(\alpha_{k})}{\Upsilon^{\prime}(\alpha_{k})},\quad k=1,2,\ldots,M_1-mb. \nn
\eee
Now, collecting the coefficient of $x^n$ from both the sides of (\ref{cdsp19}), we obtain
\bea
\pi_j^+(n,b) &=& \sum\limits_{k=1}^{M_1-mb}\frac{\gamma_{k,j}}{\alpha^{n+1}_{k}},\quad n\geq0. \lb{cdsp20}
\eea
\end{description}
This completes the analysis of obtaining the departure epoch probability vectors presented in (\ref{cdsp11}), (\ref{cdsp11}), (\ref{cdsp18}), (\ref{cdsp20}).
\begin{remark} \lb{spRemark1}
In this paper we are assuming that all the roots of $\Upsilon(x)=0$ are distinct, for repeated roots slight
modification is needed for that one may refer to Pradhan and Gupta \cite{pradhan2017analysis}.
\end{remark}
\section{Some special cases}\label{sec6}
In this section we discuss some special cases of the model studied in the previous sections.
\subsection{$D$-$MAP/G/1/\infty$ queue}\label{sec6.1}
We assume that the server provides individual service to the customers, according to the order of their
arrival, i.e., $a=1,~b=1$. As a result, the question of dependency of service rate on the batch size does not
arise, hence $G_n=G$. Therefore, our model reduces to the D-$MAP/G/1/\infty$ queue.
From (\ref{dsp6}) we get the VGF of the queue content at departure epoch as
\bea \bPhi^+(x) (x \mathbf{I}-\mathbf{A}^*_1(x))  &=&  \bphi^+(0) \Big( (\overline{\mathbf{D}}) x
\mathbf{A}^*_1(x) - \mathbf{A}^*_1(x) \Big). \lb{dsp6special1} \eea
The probability vector of the queue content at arbitrary and pre-arrival epochs are given by \bea
\bp(0,0)  &=& \Omega^{-1}  \bphi^+(0)  (\mathbf{I}-\mathbf{C})^{-1}, \lb{asp1Special1} \\
\bpi(0,1) &=& \Big( \bp(0,0) \mathbf{D} + \Omega^{-1} (\bphi^+(1)-\bpi^+(0,1)) \Big)(\mathbf{I}-\mathbf{C})^{-1}, \lb{asp2Special1}  \\
\bpi(n,1) &=& \Big( \bpi(n-1,1) \mathbf{D} + \Omega^{-1} (\bphi^+(n+1)-\bpi^+(n,1)) \Big)(\mathbf{I}-\mathbf{C})^{-1},~~n \ge 1. \lb{asp5Special1}
\eea
\bea
\bp^-(0,0) &=&  \frac{\bp(0,0)\mathbf{D}}{\lambda^{*}},  \\
\bpi^-(n,1) &=&  \frac{\bpi(n,1)\mathbf{D}}{\lambda^{*}}, n \ge 0. \eea
Here $\bpi^+(n,1)=\bphi^+(n),~n \ge 0.$
One may note that the waiting-time analysis of D-$MAP/G/1$ queue can be obtained from those of Samanta
\cite{samanta2015waiting} by considering $D_n=0,~n \geq 2$.

%
\subsection{$D$-$MAP/G^b/1/\infty$}\label{sec6.2}
We assume that the server provides service to the customers in batches of fixed size say $b$, i.e., $a=b$.
Moreover, the service rate does not depend on the service batch size, i.e., $G_n=G$. Therefore our model
reduces to D-$MAP/G^b/1/\infty$ queue.
From (\ref{dsp6}) we get the VGF of the queue content at departure epoch as
\bea \bPhi^+(x) (x^b \mathbf{I}-\mathbf{A}^*(x)) &=& \sum_{n=a}^{b-1}  \bphi^+(n) \Big( x^b \mathbf{A}^*(x) -
x^{n} \mathbf{A}^*(x) \Big).  \lb{dsp6Special3} \eea
The probability vector of the queue content at arbitrary and pre-arrival epochs are given by \bea
\bp(n,0)  &=& \Omega^{-1} \sum_{j=0}^{n} \bphi^+(j) (\overline{\mathbf{D}})^{(n-j)} (\mathbf{I}-\mathbf{C})^{-1},~~~ 0 \le n \le b-1. \lb{asp1Special3} \\
\bpi(0,b) &=& \Big( \bp(b-1,0) \mathbf{D} + \Omega^{-1} (\bphi^+(b)-\bpi^+(0,b)) \Big)(\mathbf{I}-\mathbf{C})^{-1}, \lb{asp2Special3}  \\
\bpi(n,b) &=& \Big( \bpi(n-1,b) \mathbf{D} + \Omega^{-1} (\bphi^+(n+b)-\bpi^+(n,b)) \Big)(\mathbf{I}-\mathbf{C})^{-1},~~n \ge 1. \lb{asp5Special3}
\eea
\bea
\bp^-(n,0) &=&  \frac{\bp(n,0)\mathbf{D}}{\lambda^{*}},~~~~ 0 \le n \le b-1,  \\
\bpi^-(n,b) &=&  \frac{\bpi(n,b)\mathbf{D}}{\lambda^{*}},~~~~ n \ge 0.
\eea
\begin{remark}
In the case of D-$MAP/G/1/\infty$ and D-$MAP/G^b/1/\infty$, only the VGF of the queue length can be obtained.
Since in both the cases server serves only a fixed number of customers.
\end{remark}
\subsection{$D$-$MAP/G^{(a,b)}/1/\infty$}\label{sec6.3}
Although, the finite-buffer D-$MAP/G^{(a,b)}/1/N$ queue has been studied by Chaudhry and Gupta
\cite{chaudhry2003analysis}, the corresponding infinite-buffer D-$MAP/G^{(a,b)}/1/\infty$ queue has not been
considered so far in the literature. The results of this model can be obtained by dropping the batch-size
dependency service in our model, i.e., we assume $G_n=G$.
From (\ref{dsp6}) we get the VGF of the queue content at departure epoch as
\bea
\bPhi^+(x) (x^b \mathbf{I}-\mathbf{A}^*(x)) &=& \sum_{n=0}^{a-1} \bphi^+(n) \Big( (\overline{\mathbf{D}})^{(a-n)} x^b
\mathbf{A}^*(x) - x^{n} \mathbf{A}^*(x) \Big) \nn \\
&+&  \sum_{n=a}^{b-1}  \bphi^+(n) \Big( x^b \mathbf{A}^*(x) - x^{n} \mathbf{A}^*(x) \Big).  \lb{dsp6Special2}
\eea Joint VGF of the queue and server content distribution at departure epoch is given by \bea
\widetilde{\bpi}^+(x,\xi)
&=& \Bigg( \sum_{n=0}^{a-1} \bphi^+(n) \Big( (\xi^b-\xi^a) (\overline{\mathbf{D}})^{(a-n)} \mathbf{A}^*(x) \mathbf{A}^*(x) + (\overline{\mathbf{D}})^{(a-n)} \xi^a x^b  \mathbf{A}^*(x) - \xi^b x^n  \mathbf{A}^*(x) \Big) \nn \\
&+&  \sum_{n=a}^{b-1}  \bphi^+(n) \Big( (\xi^b-\xi^n) \mathbf{A}^*(x)\mathbf{A}^*(x) + \xi^n x^b \mathbf{A}^*(x) - \xi^b x^{n} \mathbf{A}^*(x) \Big)\Bigg)(x^b \mathbf{I}-\mathbf{A}^*(x))^{-1}. \nn \\
 \lb{dsp9Special2}
\eea Joint queue and server content distribution at arbitrary and pre-arrival epochs is same as given in
(\ref{asp1})-(\ref{asp5}) and (\ref{asp12})-(\ref{asp13}), respectively.
\section{Performance measures}\label{sec7}
Having found the probability vectors $\textbf{p}(n,0)$, ($0\leq n
\leq a-1$), $\bpi(n,r)$, ($a\leq r \leq b$, $n\geq 0$), the other
significant distributions of interest can be easily obtained and are
given below.
\begin{itemize}
 \item Distribution of the number of customers in the system at an arbitrary epoch
 (including number of customers with the server) is given by\\
$ p_n^{system} = \left\{\begin{array}{r@{\mskip\thickmuskip}l}
& \textbf{p}(n,0)\bec\hspace{3.8cm} 0\leq n \leq a-1,\\
& \sum_{r=a}^{min(b,n)}\bpi(n-r,r)\bec\hspace{1.5cm} a\leq n \leq b,\vspace{0.3cm}\\
&\sum_{r=a}^{b}\bpi(n-r,r)\bec\hspace{2.2cm} n\geq b+1.
\end{array}\right.$\\

\item Distribution of the number of customers in the queue at arbitrary epoch is given by\\
$ p_n^{queue} =\left\{\begin{array}{r@{\mskip\thickmuskip}l}
& \textbf{p}(n,0)\bec+\sum_{r=a}^{b}\bpi(n,r)\bec\hspace{1.4cm} 0\leq n \leq a-1,\\
&\sum_{r=a}^{b}\bpi(n,r)\bec\hspace{3.2cm} n\geq a.
\end{array}\right.
$

\item Distribution of the number of customers in service given that server is busy
\bea p_r^{server} = c\sum_{n=0}^{\infty}\bpi(n,r)\bec,~~  a\leq r
\leq b\eea where
$c^{-1}=\left[1-\sum_{n=0}^{a-1}\textbf{p}(n,0)\bec\right]=P_{busy}$
\end{itemize}
It is very much essential to study the performance measures of the
queueing system as they play a notable role in designing and
improving the efficiency of the system. Some performance measures
are listed below:
\begin{itemize}
\item average number of customers waiting in the queue ($L_q)=\sum_{n=0}^{\infty}n p_n^{queue}$,
\item average number of customers in the system ($L)=\sum_{n=0}^{\infty}n p_n^{system}$,
\item average number of customers with the server ($L_s)=\sum_{r=a}^{b}r p_r^{server}$,
\item average waiting time of a customer in the queue $(W_q)=\frac{L_q}{\la^*}$, as well as in the system $(W)=\frac{L}{\la^*}$.
\item the probability that the server is idle $(P_{idle})=\sum_{n=0}^{a-1}\textbf{p}(n,0)\bec$.
\end{itemize}
\section{Numerical examples}\label{sec8}
In this section, we illustrate the methodology and the results derived in previous sections through some
numerical examples which have been done using Maple 15 on PC having configuration Intel (R) Core (TM) i5-3470
CPU Processor @ 3.20 GHz with 4.00 GB of RAM. Though several results have been generated, a few of them are
presented here which may be useful to researchers and practitioners. Numerical results for two different
service-time distributions viz. discrete phase-type and negative binomial are given in the following
examples. All the results are presented in 6 decimal for sake of brevity.
\begin{example}
The discrete phase($DPH$)-type service time distribution.
\end{example}
\nii
In this example the D-$MAP$ is represented by the matrices \\
$\textbf{C}=\bordermatrix{ &    &   &   \cr
                        & 0.30   &   0.10  & 0.15 \cr
                          & 0.35   & 0.05  & 0.20 \cr
                          & 0.15   & 0.10  & 0.15 \cr}$
                          and
                          $\textbf{D}=\bordermatrix{ &    &  &   \cr
                          & 0.10   & 0.25 & 0.10  \cr
                          & 0.20   & 0.15 & 0.05 \cr
                          & 0.45   & 0.05 & 0.10 \cr}$,\\ \\
that gives  $\la^*=0.474456$, $\overline{\bpi}=[0.489130,  0.260869, 0.2500000]$.
The $DPH$-type distribution has the representation ($\boldsymbol{\beta},~\textbf{T}$), where
$\boldsymbol{\beta}$ is a row vector of order $\nu$ and $\textbf{T}$ is a square matrix of order $\nu$. The
parameters chosen are $a=6$, $b=10$, $m=3$, and the batch-size dependent service time distribution for
$DPH_n$($\boldsymbol{\beta}_n$,$\textbf{T}_n$), $6\leq n \leq 10$ is given in the following table.\\~ \\
                          $\begin{tabular}{|c|c|c|c|}\hline
                             batch size ($n$)& $\boldsymbol{\beta}_n$ & $\mathbf{T}_n$ & $S_n$ \\\hline
                             6& $\left(
                                  \begin{array}{ccc}
                                    0.3 & 0.4 & 0.3 \\
                                  \end{array}
                                \right)$
                              & $\left(
                                  \begin{array}{ccc}
                                    0.7 & 0.2 & 0.1 \\
                                    0.2 & 0.6 & 0.1 \\
                                    0.1 & 0.0 & 0.8 \\
                                  \end{array}
                                \right)$
                              & 15.60 \\\hline
                             7& $\left(
                                  \begin{array}{ccc}
                                    0.5 & 0.3 & 0.2 \\
                                  \end{array}
                                \right)$ & $\left(
                                  \begin{array}{ccc}
                                    0.7 & 0.1 & 0.1 \\
                                    0.1 & 0.7 & 0.1 \\
                                    0.1 & 0.2 & 0.6 \\
                                  \end{array}
                                \right)$
                                & 10.00\\\hline
                             8& $\left(
                                  \begin{array}{ccc}
                                    0.4 & 0.2 & 0.4 \\
                                  \end{array}
                                \right)$ & $\left(
                                  \begin{array}{ccc}
                                    0.8 & 0.0 & 0.1 \\
                                    0.1 & 0.6 & 0.1 \\
                                    0.0 & 0.1 & 0.8 \\
                                  \end{array}
                                \right)$
                                & 8.461538\\\hline
                             9& $\left(
                                  \begin{array}{cccc}
                                    0.25 & 0.25 & 0.25&0.25 \\
                                  \end{array}
                                \right)$ & $\left(
                                  \begin{array}{cccc}
                                    0.4 & 0.0 & 0.0 & 0.5\\
                                    0.0 & 0.7 & 0.2 &0.0\\
                                    0.0 & 0.0 & 0.5 &0.3\\
                                    0.2 & 0.1 & 0.0 &0.6\\
                                  \end{array}
                                \right)$
                                & 8.678161\\\hline
                             10& $\left(
                                  \begin{array}{cccc}
                                    0.3 & 0.1 & 0.2&0.4 \\
                                  \end{array}
                                \right)$ & $\left(
                                  \begin{array}{cccc}
                                    0.6 & 0.0 & 0.0&0.3 \\
                                    0.0 & 0.6 & 0.1 &0.0\\
                                    0.0 & 0.0 & 0.5 &0.2\\
                                    0.1 & 0.1 & 0.0 &0.7\\
                                  \end{array}
                                \right)$
                                &6.604651\\\hline
                          \end{tabular}$\\
                          ~\\
So $\rho=0.313361$. The joint queue and server content distribution for D-$MAP/PH_n^{(6,10)}/1$ queue, at
different epochs has been displayed in Tables \ref{mapbsd1} and \ref{mapbsd2}.
\begin{example}
Negative binomial (NB) service time distribution
\end{example}
\nii The matrices correspond to the D-$MAP$ are given
by\\
$\textbf{C}=\bordermatrix{ &    & &    \cr
                          &0.4   &   0.1 &0.05  \cr
                          & 0.25  & 0.05 &0.30  \cr
                          &0.10 & 0.15 & 0.15\cr}$
 and $\textbf{D}=\bordermatrix{ &    &  &   \cr
                          & 0.15   & 0.20  &0.10 \cr
                          & 0.15   & 0.20  &0.05 \cr
                          &0.05 & 0.45 & 0.10\cr}$. \\ \\
That gives $\la^*=0.469067$, $\overline{\bpi}=[0.398305,0.355932, 0.245762]$. The input parameters chosen are
$a=4$, $b=7$, $m=3$, and mean service times of NB distribution for batch size dependent service
time distributions are 
taken as
\begin{tabular}
{|c|c|}\hline
                batch size ($r$)& $S_r$  \\\hline
                 4&  3  \\
                 5&  4 \\
                 6&  5.666667  \\
                 7&   9 \\\hline
\end{tabular}.
So $\rho=0.603087$. The joint queue and server content distribution for D-$MAP/NB_n^{(4,7)}/1$ queue, at different
epochs has been displayed in Tables \ref{mapbsd3} - \ref{mapbsd4}.\\
\begin{sidewaystable}
$\vspace{0.1cm}$
\caption{Joint distribution of queue and server content and phase of the arrival process at \\
departure epoch for D-MAP$/G^{(6,10)}_n/1$ queue with $G\sim DPH$} \lb{mapbsd1}
\begin{tiny}
\begin{tabular}{|c|ccc|ccc|ccc|ccc|ccc|c|}\hline
&\multicolumn{3}{c|}{$r=6$}&\multicolumn{3}{c|}{$r=7$}&\multicolumn{3}{c|}{$r=8$}&\multicolumn{3}{c|}{$r=9$}
&\multicolumn{3}{c|}{$r=10$}&\\\hline
$n$ & $\pi^+_1(n,6)$ &$\pi^+_2(n,6)$& $\pi^+_3(n,6)$ & $\pi^+_1(n,7)$ & $\pi^+_2(n,7)$ & $\pi^+_3(n,7)$ &$\pi^+_1(n,8)$ & $\pi^+_2(n,8)$&$\pi^+_3(n,8)$  & $\pi^+_1(n,9)$ & $\pi^+_2(n,9)$ &$\pi^+_3(n,9)$& $\pi^+_1(n,10)$ & $\pi^+_2(n,10)$ &$\pi^+_3(n,10)$&$\boldsymbol{\psi}_n^+\be$\\\hline
0&0.020431&0.006413&0.012075&0.002774&0.000914&0.001657&0.002860&0.000941&0.001707&0.002569&0.000844&0.001533&0.002786&0.000913&0.001662&0.060079\\
1&0.032679&0.018175&0.016415&0.004589&0.002507&0.002295&0.004607&0.002545&0.002293&0.004064&0.002263&0.002016&0.006699&0.003206&0.003555&0.107908\\
2&0.027662&0.015246&0.013953&0.003691&0.002066&0.001852&0.003541&0.002000&0.001771&0.003045&0.001726&0.001520&0.008872&0.004530&0.004608&0.096083\\
3&0.024095&0.013240&0.012166&0.002991&0.001672&0.001499&0.002766&0.001558&0.001382&0.002359&0.001329&0.001178&0.009975&0.005227&0.005131&0.086568\\
4&0.021124&0.011602&0.010668&0.002422&0.001354&0.001214&0.002164&0.001219&0.001081&0.001847&0.001040&0.000923&0.010354&0.005503&0.005296&0.077811\\
5&0.018528&0.010177&0.009356&0.001961&0.001096&0.000983&0.001693&0.000954&0.000846&0.001454&0.000818&0.000727&0.010246&0.005494&0.005223&0.069556\\
15&0.004859&0.002671&0.002453&0.000237&0.000132&0.000119&0.000138&0.000078&0.000069&0.000135&0.000076&0.000067&0.003898&0.002138&0.001969&0.019039\\
30&0.000645&0.000354&0.000325&0.000010&0.000005&0.000005&0.000003&0.000001&0.000001&0.000003&0.000002&0.000001&0.000532&0.000292&0.000268&0.002447\\
50&0.000038&0.000024&0.000022&0.000000&0.000000&0.000000&0.000000&0.000000&0.000000&0.000000&0.000000&0.000000&0.000035&0.000019&0.000018&0.000156\\
70&0.000002&0.000001&0.000001&0.000000&0.000000&0.000000&0.000000&0.000000&0.000000&0.000000&0.000000&0.000000&0.000002&0.000001&0.000001&0.000008\\
$\geq 80$&0.000000&0.000000&0.000000&0.000000&0.000000&0.000000&0.000000&0.000000&0.000000&0.000000&0.000000&0.000000&0.000000&0.000000&0.000000&0.000000\\\hline
Total&0.273939&0.145986&0.139975&0.026778&0.014280&0.013686&0.023617&0.012594&0.012070&0.020771&0.011077&0.010616&0.144106&0.076847&0.073651&1.000000\\\hline
\end{tabular}
\end{tiny}
$\vspace{0.2cm}$
\caption{Joint distribution of queue and server content and phase of the arrival process at \\
arbitrary epoch for D-MAP$/G^{(5,9)}_n/1$ queue with $G\sim DPH$}\lb{mapbsd2}
\begin{tiny}
\begin{tabular}{|c|ccc|ccc|ccc|ccc|c|}\hline
&\multicolumn{3}{c|}{$r=0$}&\multicolumn{3}{c|}{$r=6$}&\multicolumn{3}{c|}{$r=7$}&\multicolumn{3}{c|}{$r=8$}\\\hline
$n$ & $p_1(n,0)$ & $p_2(n,0)$ &$p_3(n,0)$&$\pi_1(n,6)$ & $\pi_2(n,6)$ &$\pi_3(n,6)$ & $\pi_1(n,7)$ & $\pi_2(n,7)$ &$\pi_3(n,7)$ & $\pi_1(n,8)$ & $\pi_2(n,8)$ &$\pi_3(n,8)$ \\\hline
0&0.004053&0.001353&0.002428&0.034176&0.018295&0.017416&0.003284&0.001453&0.001805&0.002843&0.001264&0.001560\\
1&0.011169&0.005138&0.006011&0.029689&0.016284&0.014988&0.002627&0.001435&0.001314&0.002187&0.001204&0.001090\\
2&0.017446&0.008584&0.009196&0.025955&0.014263&0.013106&0.002113&0.001183&0.001061&0.001696&0.000957&0.000848\\
3&0.023115&0.011676&0.012064&0.022718&0.012484&0.011470&0.001713&0.000957&0.000858&0.001324&0.000746&0.000661\\
4&0.028207&0.014459&0.014641&0.019880&0.010925&0.010037&0.001386&0.000775&0.000695&0.001033&0.000582&0.000516\\
5&0.032759&0.016950&0.016943&0.017390&0.009558&0.008779&0.001123&0.000628&0.000562&0.000805&0.000454&0.000402\\
15&&&&0.004533&0.002492&0.002288&0.000136&0.000076&0.000068&0.000064&0.000036&0.000032\\
30&&&&0.000602&0.000331&0.000303&0.000005&0.000003&0.000003&0.000001&0.000000&0.000000\\
50&&&&0.000040&0.000022&0.000020&0.000000&0.000000&0.000000&0.000000&0.000000&0.000000\\
70&&&&0.000002&0.000001&0.000001&0.000000&0.000000&0.000000&0.000000&0.000000&0.000000\\
$\geq 80$&&&&0.000000&0.000000&0.000000&0.000000&0.000000&0.000000&0.000000&0.000000&0.000000\\\hline
Total&0.489130&0.260869&0.250000&0.270667&0.148246&0.136812&0.017029&0.009107&0.008694&0.012707&0.006800&0.006486\\\hline
\end{tabular}
\end{tiny}
$\vspace{0.2cm}$
\begin{tiny}
\begin{tabular}{|c|ccc|ccc|c|}\hline
&\multicolumn{3}{c|}{$r=9$}&\multicolumn{3}{c|}{$r=10$}&\\\hline
$n$ & $\pi_1(n,9)$ & $\pi_2(n,9)$ &$\pi_3(n,9)$& $\pi_1(n,10)$ & $\pi_2(n,10)$& $\pi_3(n,10)$&\\\hline
0&0.002493&0.001110&0.001368&0.002119&0.000951&0.001159&0.099130\\
1&0.001914&0.001054&0.000954&0.003374&0.001680&0.001767&0.103879\\
2&0.001491&0.000840&0.000746&0.004055&0.002102&0.002094&0.107736\\
3&0.001175&0.000661&0.000587&0.004353&0.002300&0.002231&0.111093\\
4&0.000926&0.000521&0.000463&0.004391&0.002346&0.002241&0.114024\\
5&0.000730&0.000411&0.000365&0.004262&0.002294&0.002169&0.116584\\
15&0.000068&0.000038&0.000034&0.001537&0.000843&0.000776&0.013021\\
30&0.000001&0.000001&0.000000&0.000208&0.000114&0.000105&0.001677\\
50&0.000000&0.000000&0.000000&0.000014&0.000007&0.000007&0.000110\\
70&0.000000&0.000000&0.000000&0.000000&0.000000&0.000000&0.000004\\
$\geq 80$&0.000000&0.000000&0.000000&0.000000&0.000000&0.000000&0.000000\\\hline
Total&0.011462&0.006133&0.005850&0.060510&0.032418&0.030869&1.000000\\\hline
\multicolumn{8}{|c|}{$L$=11.399998,~~ $L_q$=6.164917,~~ $L_{s}$=6.854090}\\
\multicolumn{8}{|c|}{}\\
\multicolumn{8}{|c|}{$P_{idle}$=0.236203,~~ $W$=24.027487,~~ $W_q$=12.993639}\\\hline
\end{tabular}
\end{tiny}
\end{sidewaystable}
\begin{sidewaystable}
$\vspace{0.1cm}$
\caption{Joint distribution of queue and server content and phase of the arrival process at \\departure epoch for D-MAP$/G^{(4,7)}_n/1$ queue, with $G\sim $NB}  \lb{mapbsd3}
\begin{tiny}
\begin{tabular}{|c|ccc|ccc|ccc|ccc|c|}\hline
&\multicolumn{3}{c|}{$r=4$}&\multicolumn{3}{c|}{$r=5$}&\multicolumn{3}{c|}{$r=6$}&\multicolumn{3}{c|}{$r=7$}&\\\hline
$n$ & $\pi^+_1(n,4)$ & $\pi^+_2(n,4)$ & $\pi^+_3(n,4)$  & $\pi^+_1(n,5)$ & $\pi^+_2(n,5)$  & $\pi^+_3(n,5)$  & $\pi^+_1(n,6)$ & $\pi^+_2(n,6)$ &$\pi^+_3(n,6)$
& $\pi^+_1(n,7)$ & $\pi^+_2(n,7)$ & $\pi^+_3(n,7)$ & $\boldsymbol{\psi}_n^+\be$ \\\hline
0&0.121956&0.042343&0.076275&0.001589&0.000563&0.000865&0.000425&0.000151&0.000224&0.000094&0.000033&0.000047&0.244565\\
1&0.134098&0.144599&0.087626&0.002001&0.001966&0.001316&0.000626&0.000557&0.000402&0.000204&0.000146&0.000122&0.373663\\
2&0.073266&0.086478&0.047233&0.001404&0.001522&0.000904&0.000546&0.000541&0.000347&0.000268&0.000224&0.000165&0.212898\\
3&0.032838&0.040127&0.021078&0.000813&0.000914&0.000521&0.000398&0.000411&0.000253&0.000285&0.000254&0.000177&0.098069\\
4&0.013417&0.016701&0.008592&0.000430&0.000493&0.000275&0.000266&0.000281&0.000169&0.000271&0.000250&0.000169&0.041314\\
5&0.005194&0.006540&0.003321&0.000215&0.000251&0.000138&0.000169&0.000181&0.000108&0.000241&0.000227&0.000150&0.016735\\
10&0.000031&0.000040&0.000019&0.000004&0.000005&0.000003&0.000012&0.000013&0.000007&0.000087&0.000085&0.000054&0.000360\\
15&0.000000&0.000000&0.000000&0.000000&0.000000&0.000000&0.000000&0.000000&0.000000&0.000024&0.000024&0.000015&0.000063\\
20&0.000000&0.000000&0.000000&0.000000&0.000000&0.000000&0.000000&0.000000&0.000000&0.000006&0.000006&0.000004&0.000016\\
$\geq30$&0.000000&0.000000&0.000000&0.000000&0.000000&0.000000&0.000000&0.000000&0.000000&0.000000&0.000000&0.000000&0.000000\\\hline
Total&0.383811&0.340657&0.246069&0.006651&0.005942&0.004147&0.002685&0.002400&0.001667&0.002374&0.002122&0.001468&1.000000\\\hline
\end{tabular}

$\vspace{0.1cm}$
\caption{Joint distribution of queue and server content and phase of the arrival process at \\arbitrary epoch for D-MAP$/G^{(4,7)}_n/1$ queue, with $G\sim $NB}\lb{mapbsd4}

\begin{tabular}{|c|ccc|ccc|ccc|ccc|ccc|c|}\hline
&\multicolumn{3}{c|}{$r=0$}&\multicolumn{3}{c|}{$r=4$}&\multicolumn{3}{c|}{$r=5$}&\multicolumn{3}{c|}{$r=6$}&\multicolumn{3}{c|}{$r=7$}&\\\hline
$n$ & $p_1(n,0)$ & $p_2(n,0)$&$p_3(n,0)$ &$\pi_1(n,4)$ & $\pi_2(n,4)$ & $\pi_3(n,4)$ & $\pi_1(n,5)$ & $\pi_2(n,5)$  & $\pi_3(n,5)$ & $\pi_1(n,6)$ & $\pi_2(n,6)$ &$\pi_3(n,6)$ &$\pi_1(n,7)$&$\pi_2(n,7)$& $\pi_3(n,7)$ & $p_n^{queue}$ \\\hline
0&0.031317&0.011129&0.016323&0.065014&0.073832&0.042988&0.001416&0.001094&0.000858&0.000626&0.000459&0.000373&0.000282&0.000197&0.000165&0.246073\\
1&0.064769&0.047618&0.039072&0.031897&0.038565&0.020591&0.000751&0.000853&0.000497&0.000400&0.000421&0.000262&0.000339&0.000295&0.000210&0.246540\\
2&0.083587&0.069883&0.051231&0.013668&0.016862&0.008760&0.000412&0.000473&0.000265&0.000270&0.000286&0.000172&0.000330&0.000305&0.000206&0.246710\\
3&0.092196&0.080329&0.056750&0.005436&0.006804&0.003478&0.000210&0.000244&0.000134&0.000171&0.000184&0.000109&0.000293&0.000278&0.000184&0.246800\\
4&&&&0.002068&0.002614&0.001322&0.000103&0.000120&0.000066&0.000105&0.000113&0.000067&0.000248&0.000239&0.000156&0.007221\\
5&&&&0.000763&0.000971&0.000487&0.000049&0.000057&0.000031&0.000063&0.000068&0.000040&0.000204&0.000198&0.000128&0.003059\\
10&&&&0.000004&0.000005&0.000002&0.000000&0.000000&0.000000&0.000003&0.000004&0.000002&0.000062&0.000061&0.000039&0.000182\\
15&&&&0.000000&0.000000&0.000000&0.000000&0.000000&0.000000&0.000000&0.000000&0.000000&0.000016&0.000016&0.000010&0.000042\\
20&&&&0.000000&0.000000&0.000000&0.000000&0.000000&0.000000&0.000000&0.000000&0.000000&0.000004&0.000004&0.000002&0.000010\\
$\geq30$&&&&0.000000&0.000000&0.000000&0.000000&0.000000&0.000000&0.000000&0.000000&0.000000&0.000000&0.000000&0.000000&0.000000\\\hline
Total&0.271871&0.208960&0.163376&0.119276&0.140197&0.077901&0.002985&0.002893&0.001881&0.001724&0.001630&0.001080&0.002447&0.002250&0.001522&1.000000\\\hline
\multicolumn{17}{|c|}{$L$=3.012144,~~ $L_q$=1.553687,~~ $L_{s}$=4.099194}\\
\multicolumn{17}{|c|}{}\\
\multicolumn{17}{|c|}{$P_{idle}$=0.644208,~~ $W$=6.421554,~~ $W_q$=3.312288}\\\hline

\end{tabular}
\end{tiny}
\end{sidewaystable}


\clearpage

\vfill
\newpage

\section{Conclusion} \label{sec9}
In this paper we have addressed a much complicated yet significant, infinite-buffer discrete-time batch
service queue with the assumption of correlated arrival process, i.e., discrete-time Markovian arrival
process, with general batch size dependent service time distribution. We have used the supplementary variable
technique for the mathematical modeling and the pgf approach to obtain the probability vector generating
function of the joint distribution of the queue and server content at the departure epoch. The required
distribution is then extracted from the completely known generating function using the roots method, and its
relation has been established with the distribution at various epochs such as arbitrary, pre-arrival and
outside observer's epoch. We have discussed some significant characteristics along with some special cases of
the model. The computing process is explained thoroughly and some numerical results are also presented.



\end{document}